\documentclass{amsart}

\usepackage{amssymb}
\usepackage{amsmath}
\usepackage{amscd}
\usepackage{graphicx}
\usepackage{epstopdf}
\usepackage{epsfig,pinlabel}
\usepackage{lpic}

\usepackage{parskip}
\newtheorem{thm}{Theorem}[section]
\newtheorem{prop}[thm]{Proposition}
\newtheorem{lem}[thm]{Lemma}
\newtheorem{cor}[thm]{Corollary}

\theoremstyle{definition}
\newtheorem{definition}[thm]{Definition}

\theoremstyle{remark}

\newcommand{\Z}{\mathbf{Z}}  
\newcommand{\G}{\Gamma}  
\newcommand{\I}{\mathcal{I}}  
\newcommand{\K}{\mathcal{K}}  
\newcommand{\N}{\mathcal{N}}
\newcommand{\el}{\mathcal{L}}
\newcommand{\F}{\mathbf{F}}

\begin{document}

\title{Two Mod-p Johnson Filtrations}

\author{James Cooper}
\address{Rice University}
\email{jmc10@rice.edu}
\urladdr{math.rice.edu/~jmc10}

\begin{abstract}
We consider two mod-p central series of the free group given by Stallings and Zassenhaus. Applying these series to definitions of Dennis Johnson's filtration of the mapping class group we obtain two mod-p Johnson filtrations. Further, we adapt the definition of the Johnson homomorphisms to obtain mod-p Johnson homomorphisms.

We calculate the image of the first of these homomorphisms. We give generators for the kernels of these homomorphisms as well. We restrict the range of our mod-p Johnson homomorphisms using work of Morita. We finally prove the announced result of Perron that a rational homology 3-sphere may be given as a Heegaard splitting with gluing map coming from certain members of our mod-p Johnson filtrations.
\end{abstract}

\maketitle

\section{Introduction}

The mapping class group $Mod_{g,1}$ of an orientable surface of genus 
$g$ with 1 boundary component $\Sigma_{g,1}$ is the quotient $Homeo^+(\Sigma_{g,1},\partial\Sigma_{g,1})/Homeo_0^+(\Sigma_{g,1})$. That is, the group of orientation preserving 
homeomorphisms of $\Sigma_{g,1}$ fixing the boundary $\partial\Sigma_{g,1}$ modulo isotopy relative 
to $\partial\Sigma_{g,1}$. In the 1980s Johnson introduced the Johnson filtration of the mapping class group and 
the Johnson homomorphisms at each  level of the filtration. We introduce mod-$p$ analogues of Johnson's filtration
and homomorphisms. We prove results about these mod-$p$ versions that are analogous to several of Johnson's foundational results. For a survey of Johnson's work see 
\cite{Johnson1}, for instance. We also adapt certain results of Morita and Pitsch regarding the Johnson filtration to our mod-$p$ versions.

\subsection{Central series} A mapping class class $f \in Mod_{g,1}$ induces a group homomorphism $f_*$ on $\G = \pi_1(\Sigma_{g,1})$.
Having a single boundary component allows us 
to choose our base point to be on the boundary, which is fixed by $f$, making $f_*$ well-defined. If 
$N\triangleleft\G$ is a normal subgroup such that the mapping class $f$ leaves $N$ invariant (i.e., $f_*(N) < N$) 
then $f_*$ descends to a homomorphism on the quotient $\G/N$. Johnson studies such induced maps on quotients by terms of the lower central series of $\G$. Letting $[-,-]$ 
denote the commutator bracket we may define the lower central series $\{G_k\}$ of a group $G$.

\begin{definition}
The lower central series $\{G_k\}$ of a group $G$ is defined recursively by $G_1 = G$ and $G_k = [G,G_{k-1}]$.
\end{definition}

We consider two other series. Fix an odd prime $p$. The first series 
$\{\G_k^Z\}$ is defined by Hans Zassenhaus in terms of the lower central series $\{\G_k\}$ of $\G$.

\begin{definition}[Zassenhaus, {\cite{Zassenhaus}}]
The Zassenhaus mod-p central series $\{\G_k^Z\}$ of a group $\G$ is defined by $\G_k^Z = \prod\limits_{ip^j\geq k}(\G_i)^{{p}^j}$.
\end{definition}

\noindent The Zassenhaus mod-$p$ central series is the fastest descending series such that:
\begin{itemize}
\item $[\G_k^Z,\G_l^Z] < \G_{k+l}^Z$
\item $(\G_k^Z)^p < \G_{pk}^Z$
\end{itemize}

The second series $\{\G_k^S\}$ is defined by John Stallings recursively, much like the lower central series.

\begin{definition}[Stallings, {\cite{Stallings}}]
The Stallings mod-p central series $\{\G_k^S\}$ of a group $\G$ is defined by $\G_1^S = \G$ and $\G_k^S = [\G,\G_{k-1}^S](\G_{k-1}^S)^p$.
\end{definition}

\noindent The Stallings mod-$p$ central series is the fastest descending central series such that:
\begin{itemize}
\item $[\G_k^S,\G_l^S] < \G_{k+l}^S$
\item $(\G_k^S)^p < \G_{k+1}^S$
\end{itemize}

Note that whereas the successive quotients $\el_k$ of the lower central series are free abelian groups, the successive 
quotients of our series are $\F_p$-vector spaces. In this way each of these series can be interpreted as a 
type of mod-p lower central series. We discuss these series more in Section \ref{Section:Series}

\subsection{Johnson filtrations} Each subgroup $\G_k$ in the lower central series of $\G$ is a characteristic 
subgroup (i.e., every $f\in Aut(\G)$ satisfies $f(\G_k) < \G_k$). We therefore have that each
mapping class $f$ induces an automorphism of the quotient $\N_k = \G/\G_{k+1}$ for each $k$. 
Johnson filters $Mod_{g,1}$ by considering for which $k$ the map $f_k$ is trivial.

\begin{definition}
The Johnson filtration is the filtration $\{\I_{g,1}(k)\}$ of $Mod_{g,1}$ where:
\[
\I_{g,1}(k) = \{f\in Mod_{g,1} | f(x) \equiv x\ \text{mod}\ \G_{k+1}\text{ for all } x \in \G\}.
\]
\end{definition}

In the same way that Johnson uses the lower central series to define the filtration $\{\I_{g,1}(k)\}$ we use our 
two mod-p central series to define two filtrations $\{\I_{g,1}^Z(k)\}$ and $\{\I_{g,1}^S(k)\}$ of $Mod_{g,1}$.

\begin{definition}
The Zassenhaus mod-p Johnson filtration is the filtration $\{\I_{g,1}^Z(k)\}$ of $Mod_{g,1}$ where:
\[
\I_{g,1}^Z(k) = \{f\in Mod_{g,1} | f(x) \equiv x\ \text{mod}\ \G_{k+1}^Z\text{ for all } x \in \G\}.
\]
\end{definition}

\begin{definition}
The Stallings mod-p Johnson filtration is the filtration $\{\I_{g,1}^S(k)\}$ of $Mod_{g,1}$ where:
\[
\I_{g,1}^S(k) = \{f\in Mod_{g,1} | f(x) \equiv x\ \text{mod}\ \G_{k+1}^S\text{ for all } x \in \G\}.
\]
\end{definition}

The first member of the Johnson filtration $\I_{g,1}(1)$ is the Torelli group $\I_{g,1}$, 
which is the subgroup of $Mod_{g,1}$ that acts trivially on $H = H_1(\Sigma_{g,1};\Z)$. We note one property in 
particular: $\I_{g,1} < Mod_{g,1}$ is of infinite index.

The first term of each of the mod-p filtrations is the level-p congruence subgroup 
$Mod_{g,1}[p] = \I_{g,1}^Z(1) = \I_{g,1}^S(1)$, which is the subgroup of $Mod_{g,1}$ consisting of mapping 
classes that act trivially on $H_1(\Sigma_{g,1}; \F_p)$. Unlike $\I_{g,1}$, the subgroup 
$Mod_{g,1}[p]$ is finite index in $Mod_{g,1}$.

\subsection{Johnson homomorphisms} The Johnson filtration of $Mod_{g,1}$ is a filtration by subgroups 
consisting of mapping classes that act trivially on the quotients $\N_k$. Equivalently, $f \in \I_{g,1}(k)$ if and 
only if  $f(x)x^{-1} \in \G_{k+1}$ for all $x \in \G$. Let $[x]$ denote the homology class of $x$ in $H$.
Let $\el_{k} = \G_{k}/\G_{k+1}$. For fixed $f\in \I_{g,1}(k)$ Johnson proved that the map
\begin{align*}
 H &\longrightarrow \el_{k+1}\\
 [x] &\longmapsto f(x)x^{-1}\text{ mod }\G_{k+2}
\end{align*}
\noindent is a well-defined homomorphism.

\begin{definition}
The Johnson homomorphisms $\tau_k:\I_{g,1}(k) \longrightarrow Hom(H, \el_{k+1})$ are defined for $f\in\I_{g,1}(k)$ and $x\in\G$ by:
\[
\tau_k(f)([x]) = f(x)x^{-1}\ \text{mod}\ \G_{k+2}
\]
\end{definition}

\noindent Johnson shows these are well-defined homomorphisms. For $* = S$ or $Z$ let $\el_{k}^*$ denote $\G_{k}^*/\G_{k+1}^*$.
Our mod-p Johnson homomorphisms are defined as follows:

\begin{definition}
The Zassenhaus mod-$p$ Johnson homomorphisms $\tau_k^Z:\I_{g,1}^Z(k) \longrightarrow Hom(H_1(\Sigma_{g,1}; \F_p), \el_{k+1}^Z)$ 
are defined for $f\in\I_{g,1}^Z(k)$ and $x\in\G$ by:
\[
\tau_k^Z(f)([x]) = f(x)x^{-1}\ \text{mod}\ \G_{k+2}^Z,
\]
\noindent where $[x]$ denotes the class of $x$ in $H_1(\Sigma_{g,1}; \F_p)$.
\end{definition}

\begin{definition}
The Stallings mod-$p$ Johnson homomorphisms $\tau_k^S:\I_{g,1}^S(k) \longrightarrow Hom(H_1(\Sigma_{g,1}; \F_p), \el_{k+1}^S)$ 
are defined for $f\in\I_{g,1}^S(k)$ and $x\in\G$ by:
\[
\tau_k^S(f)([x]) = f(x)x^{-1}\ \text{mod}\ \G_{k+2}^S,
\]
\noindent where $[x]$ denotes the class of $x$ in $H_1(\Sigma_{g,1}; \F_p)$.
\end{definition}
\noindent We show that these are well-defined homomorphisms in Section \ref{Section:JohnsonFiltAndHomo}.

Johnson calculates the image of the first homomorphism $\tau_1$. By definition this image is contained in 
$Hom(H, \el_2)$. This is isomorphic to $H\otimes(\bigwedge^2H)$ under the isomorphism 
$\el_2 \longrightarrow \bigwedge^2H$ that takes the coset of $[x,y]\in\G_2$ to $[x]\wedge[y]$ and by virtue of $H$ being isomorphic to its dual.
Further, $H\otimes(\bigwedge^2H)$ contains $\bigwedge^3H$ as a subgroup via the inclusion:
\[
x\wedge y\wedge z \mapsto x\otimes(y\wedge z) + y\otimes(z\wedge x) + z\otimes(x\wedge y)
\]

\noindent Johnson shows that $\tau_1$ surjects onto this subgroup.

\begin{thm}[Johnson, {\cite{Johnson0}}]
$\tau_1(\I_{g,1}) \cong \bigwedge^3H$ for $g\geq 2$.
\end{thm}

We calculate the images of our first mod-p homomorphisms $\tau_1^Z$ and $\tau_1^S$, proving the following theorems
in Section \ref{Section:Images}.

\begin{thm}\label{Theorem:ZassenhausImage}
$image(\tau_1^Z) \cong \bigwedge^3H_1(\Sigma_{g,1};\F_p)$ for $g\geq 2$.
\end{thm}

\begin{thm}\label{Theorem:StallingsImage}
$image(\tau_1^S) \cong \bigwedge^3H_1(\Sigma_{g,1};\F_p) \oplus \mathfrak{sp}_{2g}(\F_p)$ for $g\geq 2$.
\end{thm}

The abelianization of $Mod_{g,1}[p]$ in our case of $p$ being an odd prime is known due to independent work of 
Perron in \cite{Perron2}, Putman in \cite{Putman1}, and Sato in \cite{Sato}.

\begin{thm}[Perron, {\cite{Perron2}}; Putman, {\cite{Putman1}}; Sato, {\cite{Sato}}]
For $p$ an odd prime and $g\geq 3$, $H_1(Mod_{g,1}[p]) \cong \bigwedge^3H_1(\Sigma_{g,1};\F_p)\oplus\mathfrak{sp}_{2g}(\F_p)$.
\end{thm}
\noindent Theorem \ref{Theorem:StallingsImage} shows that $\tau_1^S$ induces the abelianization of $Mod_{g,1}[p]$.

\subsection{Morita's restriction} Morita restricts the target of Johnson's homomorphisms $\tau_k$ in the following way:

\begin{thm}[Morita, {\cite{Morita}}]
$image(\tau_k) < H_3(\N_k)$
\end{thm}

\noindent Morita's argument easily generalizes to obtain a similar restriction for our mod-p Johnson homomorphisms. 
In particular, we have:

\begin{thm}
$image(\tau_k^*) < H_3(\N_k^*;\F_p)$ for $* = S$ or $Z$.
\end{thm}

\noindent In Section \ref{Section:Morita} we describe how $H_3(\N_k^*;\F_p)$ lies in $Hom(H_1(\Sigma_{g,1};\F_p), \el_{k+1}^*)$.

\subsection{Generating $ker(\tau_1^*)$} Johnson gives a nice characterization of the kernel of the first Johnson homomorphism, often denoted $\K_{g,1}$.
Let $\gamma \subset \Sigma_{g,1}$ be a simple closed curve such that its complement in $\Sigma_{g,1}$ is 
disconnected. Such a curve is called a separating curve. It is easy to see by a direct calculation that if
$\gamma$ is a separating curve then the Dehn twist $T_{\gamma}$ about $\gamma$ is a member of $\K_{g,1}$. 
We call such mapping classes separating twists. Johnson shows that these are enough to generate $\K_{g,1}$.

\begin{thm}[Johnson, {\cite{Johnson2}}]
For $g\geq 3$, $\K_{g,1} = ker(\tau_1)$ is the subgroup of $Mod_{g,1}$ generated by Dehn twists about separating curves.
\end{thm}

In Section \ref{Section:Kernels}, we describe generating sets for the kernels $\K_{g,1}^Z$ and $\K_{g,1}^S$ of $\tau_1^Z$ and $\tau_1^S$ respectively.
The generating sets contain certain subsets of three families of maps. The first family is
the collection of separating twists. The second family of maps is populated by bounding pair maps, 
which are elements of $\I_{g,1}$. These are maps $T_aT_b^{-1}$, where $a$ and $b$ are a pair of disjoint, 
homologous simple closed curves in $\Sigma_{g,1}$. The third and final family of maps consists of 
$p^{th}$-powers of Dehn twists $T_c^p$ about simple closed curves $c$.

\begin{thm}
For $g\geq 3$, $\K_{g,1}^Z = ker(\tau_1^Z)$ is generated by Dehn twists about separating curves and $p^{th}$-powers of Dehn twists.
\end{thm}

We note that Boggi and Pikaart provide an earlier, algebo-geometic proof of the above theorem (see Corollary 3.11 in \cite{Boggi}).

\begin{thm}
For $g\geq 3$, $\K_{g,1}^S = ker(\tau_1^S)$ is generated by Dehn twists about separating curves, $p^{th}$-powers of bounding pair maps, and $p^2$-powers of Dehn twists.
\end{thm}

\subsection{Perron's mod-$p$ Johnson filtration}In his research announcement \cite{Perron2}, Perron also defined a mod-$p$ version of Johnson's filtration of 
the mapping class group. His definition makes use of a mod-$p$ Magnus representation and conditions on evaluations of certain higher order Fox derivative terms. We show in Section \ref{Section:FoxCalc} 
that this filtration coincides with the Zassenhaus filtration $\{\I_{g,1}^Z(k)\}$ as we define it. This makes
use of a characterization of the Zassenhaus series and filtration in terms of Fox calculus, which we also review
in Section \ref{Section:FoxCalc}.

\subsection{Rational homology spheres} We end with Section \ref{Section:QHS}, giving a proof of an announced result of Bernard Perron:

\begin{thm}
If $M$ is a rational homology 3-sphere and $p$ is a prime relatively prime to $|H_1(M;\Z)|$ then $M$ has a Heegaard splitting $U \cup_f V$ of some genus $g$ with gluing map $f \in Mod_{g,1}[p]$.
\end{thm}

This generalizes the statement that integral homology 3-spheres are obtained analogously from mapping classes 
in $\I_{g,1}$, which also follows from our proof. To our knowledge this fills a gap in the literature.
While this statement is generally known, we do not know a published proof of it.

Pitsch shows that in fact all integral homology 3-spheres can be obtained from the term $\I_{g,1}(3)$ of the 
Johnson filtration \cite{Pitsch}. With this restriction we offer a similar restriction in 
the rational homology 3-sphere case.

\begin{thm}
If $M$ is a rational homology 3-sphere and $p$ is a prime relatively prime to $|H_1(M;\Z)|$ then $M$ has a 
Heegaard splitting $U \cup_f V$ of some genus $g$ with gluing map $f$ a product of terms in $\I_{g,1}(3)$ and 
$p^{th}$-powers of Dehn twists. In particular, $f\in \I_{g,1}^Z(3)$ if $p\geq 5$.
\end{thm}

\subsection{Outline} In Section 2 we introduce and discuss the Stallings and Zassenhaus central series.
Section 3 contains the definitions of our mod-$p$ Johnson filtrations and homomorphisms. We show that Perron's mod-$p$
filtration is equal to the Zassenhaus filtration in Section 4. In Section 5 we calculate the images of $\tau_1^S$ and
$\tau_1^Z$. Section 6 discusses how Morita's image restriction for the classical Johnson homomorphisms offers a
restriction for our mod-$p$ Johnson homomorphisms. Generating sets for $ker(\tau_1^S)$ and $ker(\tau_1^Z)$ are
obtained in Section 7. We end the paper with Section 8, where we show that all rational homology 3-spheres may
be obtained as a Heegaard splitting with gluing map in $\I_{g,1}^Z(3)$ for an appropriate odd prime $p$.

\subsection{Acknowledgments} I would like to thank my doctoral advisor Andrew Putman for introducing me to the work of Dennis Johnson and for his guidance throughout the development of this project. I would also like to thank Corey Bregman and David Cohen for several conversations that helped in the writing and exposition of this paper.


\section{Filtrations of a Free Group}\label{Section:Series}

Throughout this paper we will let $\G$ denote $\pi_1(\Sigma_{g,1})$, the fundamental group of an orientable
surface of genus $g$ with 1 boundary component. We take the basepoint to be on this boundary component. 
This group is isomorphic to the free group with $2g$ generators. When an explicit generating set is needed
it will be written $\{x_1, x_2,..., x_{2g}\}$. We will use $[g,h] = ghg^{-1}h^{-1}$ as our notation for the
commutator of $g, h \in \G$. If $G_1, G_2$ are two subgroups of $\G$ then $[G_1,G_2] < \G$ is the subgroup 
generated by commutators of elements in $G_1$ with elements in $G_2$.

We recall the definitions of each series and offer another characterization. This characterization is in terms of their behavior under commutator and $p^{th}$-power operations.

\begin{definition}
The lower central series $\{\G_k\}$ for $\G$ is defined recursively by $\G_1 = \G$ and $\G_k = [\G,\G_{k-1}]$. The lower central series is the fastest descending series with respect to the property:
\begin{itemize}
\item $[\G_k,\G_l] < \G_{k+l}$.
\end{itemize}
\end{definition}

\begin{definition}[Zassenhaus, {\cite{Zassenhaus}}]
The Zassenhaus mod-p central series $\{\G_k^Z\}$ for $\G$ is defined by $\G_k^Z = \prod\limits_{ip^j\geq k}{\G_i}^{p^j}$. This series is fastest descending series with respect to the properties:
\begin{itemize}
\item $[\G_k^Z,\G_l^Z] < \G_{k+l}^Z$
\item $(\G_k^Z)^p < \G_{pk}^Z$
\end{itemize}
\end{definition}

\begin{definition}[Stallings, {\cite{Stallings}}]
The Stallings mod-p central series $\{\G_k^S\}$ for $\G$ is defined recursively by $\G_1^S = \G$ and $\G_k^S = [\G,\G_{k-1}^S](\G_{k-1}^S)^p$. This series is the fastest descending series with respect to the properties:
\begin{itemize}
\item $[\G_k^S,\G_l^S] < \G_{k+l}^S$
\item $(\G_k^S)^p < \G_{k+1}^S$
\end{itemize}
\end{definition}

First note that these filtrations are distinct. They differ in their behavior under taking $p^{th}$-powers of elements. In particular, we have the following differences for $w\in\G$ and $w\notin\G_2$:
\begin{itemize}
\item $w^p\in\G$ and $w^p\notin\G_2$ since the image of $w^p$ in $H_1(\Sigma_{g,1};\Z) \cong \G/\G_2$ is $p[w]$, which is nontrivial when $w$ is nontrivial.
\item $w^p\in\G_2^S$ and $w^p\notin\G_3^S$. This follows from the fact that $w^p$ has trivial image in $H_1(\Sigma_{g,1};\F_p)$ and nontrivial image in $\el_2^S$. This is seen to be the case in Section \ref{Section:Images}.
\item $w^p\in\G_p^Z$ and $w^p\notin\G_{p+1}^Z$. This can be checked using a calculation similar to those given in Section \ref{Section:FoxCalc}.
\end{itemize}

Our two mod-p central series share their behavior under taking commutators. Namely, they enjoy the properties:
\begin{itemize}
\item $[\G, \G_{k-1}^Z] < \G_k^Z$
\item $[\G, \G_{k-1}^S] < \G_k^S$
\end{itemize}
\noindent We therefore have the following proposition:

\begin{prop}
The Stallings and Zassenhaus mod-p central series are indeed central, hence normal, filtrations of the group $\G$.
\end{prop}

We also note that $(\G_{k-1}^S)^p \in \G_k^S$ and $(\G_{k-1}^Z)^p \in \G_k^Z$. This together with the previous proposition gives a nice property of the quotients $\el_k^S = \G_k^S/\G_{k+1}^S$ and $\el_k^Z = \G_k^Z/\G_{k+1}^Z$. Namely, since the quotients are abelian and every element has order $p$ we have the following proposition.

\begin{prop}
The quotients $\el_k^S$ and $\el_k^Z$ are $\F_p$-vector spaces.
\end{prop}

While it is true that the two mod-p central series are distinct, we do have the following relation.
We say two series $\{G_i\}$ and $\{H_i\}$ are cofinal if for every natural number $n$ there exists a natural 
number $m$ such that $G_m < H_n$, and similarly for every $l$ there is some $k$ such that $H_k < G_l$.

\begin{prop}
The Stallings and Zassenhaus mod-p central series are cofinal.
\end{prop}

\begin{proof}
For the first inclusion we note that $\G_n^S < \G_n^Z$. To see this we use the fact $\G_2^S = \G_2^Z$ as a base case
for an induction on $n$. We may write an element $w\in\G_n^S$ as a product $\prod [x_i,y_i]\prod z_i^p$ with 
$x_i \in \G$ and $y_i, z_i \in \G_{n-1}^S$. By the inductive hypothesis, $\G_{n-1}^S < \G_{n-1}^Z$. Noting the
properties of the Zassenhaus series, we see for each $i$ that $[x_i,y_i] \in \G_n^Z$ and $z_i^p \in \G_{p(n-1)}^Z < \G_n^Z$.
We conclude that $w\in\G_n^Z$.

For the second inclusion we show for fixed $l$ that $\G_{{p}^l}^Z < \G_l^S$. Let $w \in \G_{{p}^l}^Z$ and
write $w = \prod y_k^{{p}^{{j}_k}}$ where each $y_k \in \G_{{i}_k}$ with $i_kp^{{j}_k} \geq p^l$. The behavior
of the Stallings series with respect to commutators and $p^{th}$-powers lets us see that each term $y_k^{{p}^{{j}_k}}$
lies in $\G_{{i}_k+{j}_k}^S$. To see that $w\in \G_l^S$ it therefore suffices to show that $i_k + j_k \geq l$ 
for all $k$.

We first note that $i_k\geq p^{l-{j}_k}$ so that $i_k + j_k\geq p^{l-{j}_k}+j_k$. A straight forward calculation
shows that $p^{l-{j}_k}+j_k$ is minimized as a function of $j_k$ at $j_k = l - log_p(1/ln(p))$, where it achieves
the minimum value:

\[
 \frac{1-ln(1/ln(p))}{ln(p)} + l
\]

\noindent We wish to show that this is always at least $l$. Equivalently, we want $ln(1/ln(p)) \leq 1$. This follows from the fact that $ln(1/ln(p)) = -ln(ln(p)) < 0$. \qedhere
\end{proof}


\section{Johnson style filtrations and homomorphisms}
\label{Section:JohnsonFiltAndHomo}
Now that we have our two series of $\G$ we may consider the filtrations they give of the mapping class group 
of $\Sigma_{g,1}$, which we write $Mod_{g,1}$. We will denote these filtrations by $\{\I_{g,1}^S(k)\}$ 
and $\{\I_{g,1}^Z(k)\}$ for Stallings and Zassenhaus respectively. First we recall our notation for commonly 
appearing quotients. Let $\N_k^* = \G/\G_{k+1}^*$ and $\el_k^* = \G_k^*/\G_{k+1}^*$, where $*$ is $S$ for the Stallings series, $Z$ for the Zassenhaus series, or empty for the lower central series. We now recall the definition of Johnson's original filtration.

\begin{definition}
The Johnson filtration of $Mod_{g,1}$ is given by $\I_{g,1}(k) = \{f \in Mod_{g,1} | f(x)x^{-1} \in \G_{k+1}\text{ for all }x \in \G\}$.
\end{definition}

Again we note that $\I_{g,1}(k)$ is just the subgroup of $Mod_{g,1}$ consisting of mapping classes that induce a trivial action on the quotient $\N_k$. We now offer by the appropriate substitution our mod-p filtrations, which are similarly given in terms of mapping classes that act trivially on $\N_k^Z$ and $\N_k^S$.

\begin{definition}
The Zassenhaus mod-p Johnson filtration for $Mod_{g,1}$ is $\I_{g,1}^Z(k) = \{f \in Mod_{g,1} | f(x)x^{-1} \in \G_{k+1}^Z\text{ for all } x \in \G\}$.
\end{definition}

\begin{definition}
The Stallings mod-p Johnson filtration for $Mod_{g,1}$ is $\I_{g,1}^S(k) = \{f \in Mod_{g,1} | f(x)x^{-1} \in \G_{k+1}^S\text{ for all } x \in \G\}$.
\end{definition}

Now we wish to discuss the Johnson homomorphisms associated to each filtration of $Mod_{g,1}$. Again, we begin by recalling Johnson's original construction. For each $k$ we define a map:

\begin{align*}
\widetilde{\tau}_k : &\I_{g,1}(k)\longrightarrow Hom(\G, \el_{k+1}) \\
	 &f\mapsto(x \mapsto f(x)x^{-1}\ \text{mod}\ \G_{k+2})
\end{align*}

\noindent Johnson proved that $\widetilde{\tau}_k(f)$ is indeed a homomorphism. We recall the argument below.

The quotients $\el_k$ are abelian so that by the universal property of the abelianization $H = H_1(\Sigma_{g,1};\Z)$
of $\G$ each map $\widetilde{\tau}_k(f)$ factors through $H$, allowing us to make the following definition.

\begin{definition}
The $k^{th}$ Johnson homomorphism $\tau_k:\I_{g,1}(k)\longrightarrow Hom(H, \el_{k+1})$ is the map defined by $\tau_k(f)([x]) = f(x)x^{-1}\ \text{mod}\ \G_{k+2}$.
\end{definition}

The construction of the mod-p Johnson homomorphisms similarly begins by defining a map 

\begin{align*}
\widetilde{\tau}_k^*:&\I_{g,1}^*(k) \longrightarrow Hom(\G, \el_{k+1}^*)\\
	&f \mapsto (x \mapsto f(x)x^{-1}\ \text{mod}\ \G_{k+2}^*)
\end{align*}

\noindent Below we show that the image under $\widetilde{\tau}_k^*$ of an element $f\in \I_{g,1}^*(k)$ is indeed a
homomorphism. Before stating the results needed to prove this we introduce some notation. For $f \in \I_{g,1}^*(k)$
and $x\in\G_l^*$ we write $[f,x] = x^fx^{-1}$ for $f(x)x^{-1}$. The use of this notation is justified since
expressions in $[\I_{g,1}^*(k), \G_l^*]$ satisfy commutator identities. For instance,
\begin{itemize}
\item $[f, yx] = [f, y][f, x]^y = f(y)y^{-1}yf(x)x^{-1}y^{-1} = f(y)f(x)x^{-1}y^{-1}$
\item $[fg, x] = [f, x]^g[g, x] = g(f(x)x^{-1})g(x)x^{-1} = g\circ f(x)x^{-1}$
\end{itemize}
\noindent This notation is useful for packaging the statements and proofs of the following results.

\begin{thm}[Andreadakis, {\cite{Andreadakis}, Theorem 1.1}]
The Johnson filtration behaves in the following way with respect to the commutator bracket:
\begin{enumerate}
\item $[\I_{g,1}(k), \G_l] < \G_{k+l}$
\item $[\I_{g,1}(k),\I_{g,1}(l)] < \I_{g,1}(k+l)$.
\end{enumerate}
\end{thm}

\begin{lem}[Hall, {\cite{Hall}, Three Subgroup Lemma}]
Let $A$, $B$, and $C$ be subgroups of a group $G$. If $N \triangleleft G$ is normal subgroup such that $[A,[B,C]]$ and $[B,[C,A]]$ are contained in $N$ then $[C,[A,B]]$ is also contained in $N$.
\end{lem}

The third necessary result is an adaptation of Andreadakis's work in \cite{Andreadakis} to our mod-p case.

\begin{lem}\label{lemma:JohnsonRange}
If $f \in \I_{g,1}^*(k)$ and $x \in \G_l^*$ then $f(x)x^{-1} \in \G_{k+l}^*$. Equivalently $[\I_{g,1}^*(k), \G_l^*] < \G_{k+l}^*$.
\end{lem}
\begin{proof}
We will induct on $l$. The base case $l = 1$ follows from the definition of $\I_{g,1}^*(k)$. We assume the 
lemma holds for $l$. For the inductive step we prove the lemma for Stallings and Zassenhaus independently. 
Beginning with Stallings, we consider elements of the form $[x,y]\in\G_{l+1}^S$ and $z^p \in \G_{l+1}^S$, 
where $x\in\G$ and $y, z \in\G_l^S$.

We first show that $f([x,y])[x,y]^{-1} \in \G_{k+l+1}^S$ for $f\in\I_{g,1}^S(k)$.
We show this using Andreadakis's argument. First note:

\[
f([x,y])[x,y]^{-1} = [f,[x,y]] \in [\I_{g,1}^S(k),[\G,\G_l^S]]
\]

\noindent In anticipation of using the three subgroup lemma we also note the following two inclusions.

\begin{align*}
[[\I_{g,1}^S(k),\G_l^S],\G] &< [\G_{k+l}^S,\G]\text{, by induction} \\
	&< \G_{k+l+1}^S\text{, by definition of the Stallings series}
\end{align*}
\begin{align*}
[[\I_{g,1}^S(k),\G],\G_l^S] &< [\G_{k+1}^S, \G_l^S]\text{, by induction} \\
	&< \G_{k+l+1}^S\text{, by the properties of the Stallings series} 
\end{align*}

\noindent In order to use the three subgroup lemma, the groups $\I_{g,1}^S(k)$, $\G_l^S$, and $\G$ must all be subgroups of a common group. The holomorph of $\G$ serves this purpose. The holomorph of $\G$ may be defined as $Hol(\G) := \G \rtimes Aut(\G)$ with multiplication given by:
\[
(g_1,f_1)(g_2,f_2) := (g_1f(g_2),f_1f_2)
\]
\noindent Note that any normal subgroup of $\G$ is a normal subgroup of $Hol(\G)$ and that $\I_{g,1}^S(k)$ includes into $Hol(\G)$ by the usual inclusion into $Aut(\G)$. Now the three subgroup lemma implies that $[\I_{g,1}^S(k),[\G,\G_l^S]] < \G_{k+l+1}^S$. 

By a similar consideration we see that $f(z^p)z^{-p} \in \G_{k+l+1}^S$. We show that 
$f(z^p)z^{-p} \equiv (f(z)z^{-1})^p\text{ mod }\G_{k+l+1}^S$:

\begin{align*}
f(z^p)z^{-p} &= [f, z^p] \\
	&= [f,z][f,z]^z...[f,z]^{{z}^{p-1}}\\
\end{align*}

\noindent Note that $[f,z] \in \G_{k+l}^S$ by induction, and $[f,z]^{{z}^i} \in \G_{k+l}^S$ 
for each $i = 1,...,p-1$ by normality. Furthermore, $[[f,z],z^i] \in \G_{k+l+1}^S$ so that
$[f,z] \equiv [f,z]^{{z}^i}\text{ mod }\G_{k+l+1}^S$ and $f(z^p)z^{-p} \equiv (f(z)z^{-1})^p\text{ mod }\G_{k+l+1}^S$.

Now we note that $(f(z)z^{-1})^p \in (\G_{k+l}^S)^p$ by induction, and $(\G_{k+l}^S)^p \in \G_{k+l+1}^S$ by the
properties of the Stallings series. Therefore, $f(z^p)z^{-p} \in \G_{k+l+1}^S$.

The lemma for the case of Stallings then follows by checking the statement for products. Again let $f\in\I_{g,1}^S(k)$, 
and let $w_i \in \G_{l+1}^S$. We have:

\begin{align*}
 f\left(\prod\limits_{i=1}^{n} w_i\right)\left(\prod\limits_{i=1}^{n} w_i\right)^{-1} &= \prod\limits_{i=1}^{n} f(w_i)\prod_{i=1}^{n} w_{n-i}^{-1} \\
 &= f(w_1)w_1^{-1}w_1f(w_2)w_2^{-1}w_2...\\
 &\ \ \ \ \ \ \ \ \ \ \ w_{n-1}^{-1}w_{n-1}f(w_n)w_n^{-1}w_{n-1}^{-1}...w_2^{-1}w_1^{-1}
\end{align*}

\noindent We have seen above that $f(w_i)w_i^{-1} \in \G_{k+l+1}^S$ for each $i$, and by normality we also have 
$w_{n-1}f(w_n)w_n^{-1}w_{n-1}^{-1}\in\G_{k+l+1}^S$. From this we conclude that

\[
f\left(\prod\limits_{i=1}^{n} w_i\right)\left(\prod\limits_{i=1}^{n} w_i\right)^{-1}
\]

\noindent is equivalent modulo $\G_{k+l+1}^S$ to

\[
w_1f(w_2)w_2^{-1}w_2...w_{n-2}^{-1}w_{n-2}f(w_{n-1})w_{n-1}^{-1}w_{n-2}^{-1}...w_2^{-1}w_1^{-1}
\]

\noindent and by iteration this is trivial modulo $\G_{k+l+1}^S$. The statement is then proven for the case of Stallings.

For Zassenhaus we write $t = \prod x_i^{p^j} \in \G_{l+1}^Z$ so that $x_i \in \G_m$ for some $m$ with
$mp^j \geq l + 1$. By iterative use of the identity $[f, yx] = [f, y][f, x]^y$ we see that $f(t)t^{-1}$ is a product of conjugates of elements of the form $f(x_i^{p^j})x_i^{-p^j}$. Since we only wish to show 
$f(t)t^{-1}\in \G_{k+l+1}^Z$, the element of conjugation is unimportant by the normality of $\G_{k+l+1}^Z$.
The same argument we used for the case of Stallings again works here. In particular, we still have that 
$f(x_i^{p^j})x_i^{-p^j} \equiv (f(x_i)x_i^{-1})^{p^j}\text{ mod }\G_{k+l+1}^Z$, and:

\begin{align*}
(f(x_i)x_i^{-1})^{p^j} &\in (\G_{k+m}^Z)^{p^j}\text{, since } [f, x_i] \in [\I_{g,1}(k), \G+m] < \G_{k+m} < \G_{k+m}^Z\\
	&\in \G_{(k+m)p^j}^Z,\text{ by definition of the Zassenhaus series} \\
	&\in \G_{k+l+1}^Z \qedhere
\end{align*} 
\end{proof}

\noindent In the above proof, when we checked the statement of Lemma \ref{lemma:JohnsonRange} on products of elements in $\G_{l+1}^*$ we also proved the following lemma:

\begin{lem}
For fixed $f\in\I_{g,1}^*(k)$ the following map is a homomorphism:
\begin{align*}
\G &\longrightarrow \el_{k+1}^* \\
x &\longmapsto (f(x)x^{-1}\ \text{mod}\ \G_{k+2}^*)
\end{align*}
\end{lem}

Now we may define our mod-p Johnson homomorphisms
and show they are indeed well-defined homomorphisms. Since the quotients $\el_k^*$ are $\F_p$-vector spaces 
we have that $\widetilde{\tau}_k^*(f)$ factors through $H_1(\Sigma_{g,1}; \F_p)$. This again allows the following 
definitions.

\begin{definition}
The $k^{th}$ Zassenhaus mod-p Johnson homomorphisms is the map:
\[
\tau_k^Z:\I_{g,1}^Z(k) \longrightarrow Hom(H_1(\Sigma_{g,1}; \F_p), \el_{k+1}^Z)
\]
\noindent defined by $\tau_k^Z(f)([x])= f(x)x^{-1}\ \text{mod}\ \G_{k+2}^Z$.

\end{definition}

\begin{definition}
The $k^{th}$ Stallings mod-p Johnson homomorphism is the map:
\[
\tau_k^S:\I_{g,1}^S(k) \longrightarrow Hom(H_1(\Sigma_{g,1}; \F_p), \el_{k+1}^S)
\]
\noindent defined by $\tau_k^S(f)([x])= f(x)x^{-1}\ \text{mod}\ \G_{k+2}^S$.
\end{definition}

\begin{prop}
$\tau_k^*$ is a homomorphism.
\end{prop}

\begin{proof}
Let $f$ and $g$ be elements of $\I_{g,1}^*(k)$. Let $x \in \G$ be a representative of $[x]\in H_1(\Sigma_{g,1};\F_p)$. Then
\begin{align*}
\tau_k^*(f\circ g)([x]) &= f\circ g(x)x^{-1} \text{ mod } \G_{k+2}^* \\
	&= f(g(x))f(x)^{-1}f(x)x^{-1} \text{ mod } \G_{k+2}^* \\
	&= f(g(x)x^{-1})f(x)x^{-1} \text{ mod } \G_{k+2}^*
\end{align*}

\noindent Since $g(x)x^{-1} \in \G_{k+1}^*$ and $f \in \I_{g,1}^*(k)$ Lemma \ref{lemma:JohnsonRange} gives 

\[
f(g(x)x^{-1}) \equiv g(x)x^{-1} \text{ mod } \G_{k+2}^*
\]

\noindent Therefore, $\tau_k^*(f\circ g)([x]) = \tau_k^*(g)([x])\tau_k^*(f)([x])$. \qedhere
\end{proof}

Notice that it follows from the definitions of $\tau_k^*$ that the kernels of these maps coincide with the subsequent terms in the corresponding mod-p Johnson filtrations. That is:

\begin{prop}
$ker(\tau_k^*) = \I_{g,1}^*(k+1)$
\end{prop}

\begin{proof}
This follows immediately from the definitions. Namely, $f \in \I_{g,1}^*(k+1)$ if and only if 
$f(x)x^{-1} \in \G_{k+2}^*$ for every $x \in \G$, which is equivalent to saying
$\tau_k^*(f)(x) = 0 \in \el_{k+1}^*$.\qedhere
\end{proof}


\section{Fox calculus and Zassenhaus's work}\label{Section:FoxCalc}

Perron gives a definition for a mod-p Johnson filtration $\{\I_{g,1}^P(k)\}$ in the announcement \cite{Perron2}. His definition makes use of Fox's free differential calculus, which we briefly review below. After establishing the necessary machinery due to Fox and Magnus we show that our filtration $\{\I_{g,1}^Z(k)\}$ is in fact equal to Perron's filtration $\{\I_{g,1}^P(k)\}$.

\subsection{Review of Fox calculus.} In a series of papers Fox studies the group of derivations in the
group ring $\Z[\G]$, where $\G$ is a free group of rank $n$. The relevant machinery can be found in 
\cite{Fox}. First we note that we have an evaluation map $\epsilon:\Z[\G]\longrightarrow\Z$ induced by 
the trivial group homomorphism $\G \longrightarrow 1$. Recall a derivation is a map $D:\Z[\G] \longrightarrow \Z[\G]$ that enjoy the following properties:

\begin{itemize}
\item $D(u+v) = D(u) + D(v)$
\item $D(uv) = \epsilon(v)D(u) + uD(v)$
\end{itemize}
\noindent The collection of derivations on $\Z[\G]$ forms a right module over $\Z[\G]$. 

Let $\{x_1,...,x_n\}$ be a generating set for $\G$. Fox showed that the module of derivations is generated as a right $\Z[\G]$-module by the free differentials $\partial/\partial x_j$, which are uniquely determined by the following properties:

\begin{itemize}
\item $\frac{\partial x_i}{\partial x_j} = \delta_{i,j}$
\item $\frac{\partial u + v}{\partial x_j} = \frac{\partial u}{\partial x_j} + \frac{\partial v}{\partial x_j}$
\item $\frac{\partial uv}{\partial x_j} = \epsilon(v)\frac{\partial u}{\partial x_j} + u\frac{\partial v}{\partial x_j}$
\end{itemize}

We may use these differentials to define a series for $\G$ that is known to coincide with the lower central series.

\begin{definition}
The series $\{\G_k'\}$ of $\G$ is defined by the property that $w\in\G_k'$ if and only if for all $l < k$
and all sequences $i_1,...,i_l$ with $1\leq i_j \leq n$ for $1\leq j \leq l$ we have:
\begin{align*}
\epsilon\left(\frac{\partial^l w}{\partial x_{i_l} ... \partial x_{i_1}}\right) = 0
\end{align*}
\end{definition}

\begin{prop}[Fox, {\cite{Fox},4.5}]
The lower central series $\{\G_k\}$ is equal to the series $\{\G_k'\}$ .
\end{prop}

\subsection{Magnus representations.} Magnus gives a filtration for the free group $\G$ in \cite{Magnus}. In \cite{Fox} Fox shows that Magnus's series is also equal to the lower central series of $\G$. Magnus's series is given in 
terms of the representation taking his name. The Magnus representation is a map 
$Mag:\Z[\G] \longrightarrow \Z\langle\langle w_1,...,w_{2g}\rangle\rangle$ from the group ring into the ring of formal power series in the 
noncommuting variables $\{w_i\}$ with coefficients in $\Z$. It is defined by mapping generators $x_i$ of $\G$ to the power series
$1 + w_i$. Work of Magnus shows that $Mag$ is an injection. In this way a group element $w \in \G$ corresponds to a
power series $Mag(w) \in \Z\langle\langle w_1,...,w_{2g}\rangle\rangle$, which has a well-defined notion of degree. 
A filtration $\{{\G_k}''\}$ of $\G$ is then given by the requirement that $w \in \G_k$ when $Mag(w) - 1$
is a sum of monomials of degree at least $k$. The coefficients of $Mag(w)$ can be given by evaluations of Fox's 
derivatives making it easy to see the equality of $\{{\G_k}''\}$ to $\{\G_k'\}$ and therefore $\{\G_k\}$.

Zassenhaus's series originally arises from similar considerations of a Magnus style representation.
The representation $Mag_p:\F_p[\G] \longrightarrow \F_p\langle\langle w_1,...,w_{2g}\rangle\rangle$ is given by sending 
$x_i \mapsto 1 + w_i$, extending, and reducing coefficients modulo $p$. The Zassenhaus series is in this way 
analogously characterized. In particular, we have the following result of 
Zassenhaus.

\begin{prop}[Zassenhaus, {\cite{Zassenhaus}}]
For $w\in\G$ we have $w\in\G_k^Z$ if and only if for all $l < k$ and all sequences $i_1,...,i_l$ with $1\leq i_j \leq n$ for $1\leq j \leq l$ we have:
\begin{align*}
\epsilon\left(\frac{\partial^l w}{\partial x_{i_l} ... \partial x_{i_1}}\right) \equiv 0\ \text{mod}\ p
\end{align*}
\end{prop}

\subsection{Johnson's work via Fox calculus.} Using this Fox calculus definition of the lower central series 
one may give another definition of Johnson's filtration. Let $\{x_1,...,x_{2g}\}$ be a basis for $\G$ so that 
$\partial/\partial x_i$ are the usual Fox derivatives. For a given mapping class $f\in Mod_{g,1}$ we may define its Fox matrix:
\[
B(f) = (f_{ij}) = \left(\overline{\frac{\partial f(x_j)}{\partial x_i}}\right) \in GL_{2g}(\Z[\G])
\]
\noindent Here $\sum\overline{ n_ig_i}$ denotes $\sum n_ig_i^{-1} \in \Z[\G]$ and is necessary so that:

\[
B(f\circ g) = B(f)B(g)^f = (f_{ij})(f(g_{ij})). 
\] 

\noindent The following lemma is the first step in obtaining a type of Taylor expansion for $B(f)$.

\begin{lem}[Fundamental Theorem of Fox Calculus; Fox, \cite{Fox}] Let $w \in \Z[\G]$ then:
\[
w = \epsilon(w) + \sum\limits_{i = 1}^{2g}\frac{\partial w}{\partial x_i}(x_i - 1)
\]
\end{lem}

The Fundamental Theorem of Fox Calculus allows us to write a Taylor expansion for $w \in \Z[\G]$. It is given by the representation into $\Z\langle\langle w_1,...,w_{2g}\rangle\rangle$:

\[
 w \mapsto \sum\limits_{l = 1}^{\infty}\sum\limits_{1\leq j_1,...,j_l\leq 2g}\epsilon\left(\frac{\partial^l}{\partial x_{j_l}...\partial x_{j_1}}\left(\frac{\partial w}{\partial x_i}\right)\right)w_{j_l}...w_{j_1}
\]

\noindent The Fox matrix then has a Taylor expansion $B(f) \mapsto B_0(f) + B_1(f) + ...$ given by expanding the entries of $B(f)$. Namely,

\[
B_{l-1}(f) = (b_{ik}) = \left(  \sum\limits_{1\leq j_1,...,j_l\leq 2g}\epsilon\left(\frac{\partial^l}{\partial x_{j_l}...\partial x_{j_1}}\left(\frac{\overline{\partial f(x_k)}}{\partial x_i}\right)\right)w_{j_l}...w_{j_1}  \right)
\]

\begin{definition}
The Johnson filtration is defined by:
\[
\I_{g, 1}(k) = \{f \in Mod_{g, 1} | B_0(f) = Id, B_l(f) = 0,\text{ for all }l = 1, 2,...,k-1\}.
\]
\end{definition}

\noindent This definition is known to be equivalent to the usual definition of the Johnson filtration given earlier. See, for instance, \cite{Perron1}.

Using this definition of the Johnson filtration, Perron defines a mod-p Johnson filtration in the following way. Let $\omega:\Z[G]\longrightarrow\F_p[G]$ be reduction of coefficients modulo $p$. Let $\Omega:GL_{2g}(\Z[\G])\longrightarrow GL_{2g}(\F_p[\G])$ be the map $(b_{ij})\longmapsto (\omega(b_{ij}))$. Define $B^{(p)}(f) =  \Omega(B(f))\in GL_{2g}(\F_p[\G])$. Again we have a Taylor type expansion $B^{(p)}(f) \mapsto B^{(p)}_0(f) + B^{(p)}_1(f) + ...$; the entries of $B^{(p)}_i(f)$ are equal to the entries of $B_i(f)$ with coefficients modulo $p$.

\begin{definition}[Perron, {\cite{Perron2}}]
The Perron mod-p Johnson filtration is $\I_{g,1}^P(k) = \{f \in Mod_{g,1} | B_0^{(p)}(f) = Id, B_l^{(p)}(f) = 0, l = 1, 2,..., k-1\}$.
\end{definition}

\begin{thm}
$\I_{g,1}^P(k) = \I_{g,1}^Z(k)$ for all $k$.
\end{thm}

\begin{proof}
Let $f\in \I_{g,1}^P(k)$ so that $B_0^{(p)}(f) = Id, B_l^{(p)}(f) = 0; l = 1, 2,..., k-1$. We now show that $f\in\I_{g,1}^Z(k)$.
Writing the Taylor expansion for $B(f)$, we know that:
\[
\Omega\left( \sum\limits_{1\leq j_1,...,j_l\leq 2g}\epsilon\left(\frac{\partial^l}{\partial x_{j_l}...\partial x_{j_1}}\left(\frac{\overline{\partial f(x_k)}}{\partial x_i}\right)\right)(x_{j_l}-1)...(x_{j_1} - 1) \right)= 0
\]
\noindent where $w_n$ are noncommuting variables. In particular, the coefficients are zero modulo $p$ for 
$l = 1,...,k-1$. We wish to show that $f(x)x^{-1}$ has all of its Fox derivatives of order at most $k-1$ evaluate 
to zero modulo $p$ so that $f(x)x^{-1} \in \G_k^Z$. It suffices to check this on basis elements 
$x\in\{x_1,...,x_{2g}\}$. We will need the following identity, which follows from the product rule of Fox calculus:

\[
\frac{\partial u^{-1}}{\partial x} = -u^{-1} \frac{\partial u}{\partial x}
\]

\noindent We will also need the higher-order product rule (3.2) from \cite{Fox}:

\[
\frac{\partial^l uv}{\partial x_{j_l} ...\partial x_{j_1}}= \sum_{n=1}^l\left(\frac{\partial^{l-n} u}{\partial x_{j_l}...\partial x_{j_n}}\epsilon\left(\frac{\partial^{n-1} v}{\partial x_{j_{n-1}}...\partial x_{j_1}}\right)\right) + u \frac{\partial^l v}{\partial x_{j_l}...\partial x_{j_1}}
\]

\noindent We now induct on the order of the Fox derivatives, first noting that:

\begin{align*}
\frac{\partial f(x)x^{-1}}{\partial x_i} &= \epsilon(x^{-1})\frac{\partial f(x)}{\partial x_i} + f(x)\frac{\partial x^{-1}}{\partial x_i} \\
 &= \frac{\partial f(x)}{\partial x_i} - f(x)x^{-1}\frac{\partial x}{\partial x_i} \\
\end{align*} 

\noindent This evaluates to:

\[
\begin{cases}
		\epsilon(\frac{\partial f(x)}{\partial x_i}) - 1, & \text{if } x = x_i\\
		\epsilon(\frac{\partial f(x)}{\partial x_i}), & \text{if } x \neq x_i
\end{cases}
\]

\noindent Since $B_0^{(p)}(f) = Id$ we know that the above is zero in both cases. In particular, we satisfy the base case for our induction.

We now assume the inductive hypothesis and consider the following equalities:

\begin{align*}
\frac{\partial^l f(x)x^{-1}}{\partial x_{j_l}...\partial x_{j_1}} &= \frac{\partial^{l-1}}{\partial x_{j_l}...\partial x_{j_2}}\left(\frac{\partial f(x)x^{-1}}{\partial x_{j_1}}\right) \\
	&= \frac{\partial^{l-1}}{\partial x_{j_l}...\partial x_{j_2}}\left(\epsilon(x^{-1})\frac{\partial f(x)}{\partial x_{j_1}} + f(x)\frac{\partial x^{-1}}{\partial x_{j_1}}\right) \\
\end{align*}

\noindent Applying the evaluation map and using the higher order product rule (3.2) in \cite{Fox} we obtain:

\begin{align*}
\epsilon\left(\frac{\partial^l f(x)x^{-1}}{\partial x_{j_l}...\partial x_{j_1}}\right) &= \epsilon\left(\frac{\partial^{l-1}}{\partial x_{j_l}...\partial x_{j_2}}\left(\frac{\partial f(x)}{\partial x_{j_1}}\right)\right) - \epsilon\left(\frac{\partial^{l-1}}{\partial x_{j_l}...\partial x_{j_2}}\left(f(x)x^{-1}\frac{\partial x}{\partial x_{j_1}}\right)\right) \\
	&= -\epsilon\left(\frac{\partial^{l-1}}{\partial x_{j_l}...\partial x_{j_2}}\left(f(x)x^{-1}\frac{\partial x}{\partial x_{j_1}}\right)\right) \text{, since } f\in\I_{g,1}^P(k)\\
	&= -\epsilon\left(\sum_{n=2}^l\left(\frac{\partial^{l-n} f(x)x^{-1}}{\partial x_{j_l}...\partial x_{j_n}}\epsilon\left(\frac{\partial^{n-1}x}{\partial x_{j_{n-1}}...\partial x_{j_1}}\right)\right)\right) \\
	&\ \ \ \ \ \ - \epsilon\left(f(x)x^{-1}\frac{\partial^l x}{\partial x_{j_l}...\partial x_{j_1}}\right) \text{, by (3.2) in \cite{Fox}} \\
	&= -\epsilon\left(\frac{\partial^{l-2} f(x)x^{-1}}{\partial x_{j_l}...\partial x_{j_2}}\epsilon\left( \frac{\partial x}{\partial x_{j_1}} \right)\right) 
\end{align*}

\noindent The last equality holds since any higher order derivative of the generator $x$ is trivial. The last term is zero modulo $p$ by the inductive hypothesis. We therefore see that $\I_{g,1}^P(k) < \I_{g,1}^Z(k)$.

We now seek to show the other containment. Let $f\in\I_{g,1}^Z(k)$ so that the order $l \leq k-2$ partial derivatives 
of $f(x)x^{-1}$ evaluate to zero modulo $p$. That is to say the following is zero modulo $p$:

\[
\epsilon\left(\frac{\partial^{l}f(x)x^{-1}}{\partial x_{j_l}...\partial x_{j_1}}\right)
\]

\noindent We saw above that this is equal to:

\begin{align*}
\epsilon\left(\frac{\partial^{l-1}}{\partial x_{j_l}...\partial x_{j_2}}\left(\frac{\partial f(x)}{\partial x_{j_1}}\right)\right)
	&- \epsilon\left(\sum_{n=2}^{l}\left(\frac{\partial^{l-n} f(x)x^{-1}}{\partial x_{j_{l}}...\partial x_{j_n}}\epsilon\left(\frac{\partial^{n-1}x}{\partial x_{j_{n-1}}...\partial x_{j_1}}\right)\right)\right) \\
	&- \epsilon\left(f(x)x^{-1}\frac{\partial^{l} x}{\partial x_{j_{l}}...\partial x_{j_1}}\right)  
\end{align*}

\noindent Considering only the case when $x$ is a generator, for $l \geq 2$ this again simplifies to:

\[
\epsilon\left(\frac{\partial^{l-1}}{\partial x_{j_l}...\partial x_{j_2}}\left(\frac{\partial f(x)}{\partial x_{j_1}}\right)\right) - \epsilon\left(\frac{\partial^{l-2} f(x)x^{-1}}{\partial x_{j_l}...\partial x_{j_2}}\epsilon\left( \frac{\partial x}{\partial x_{j_1}} \right)\right)
\]

\noindent Our hypothesis on the evaluation modulo $p$ of derivatives of $f(x)x^{-1}$ lets us conclude that
 
\[
\epsilon\left(\frac{\partial^{l-1}}{\partial x_{j_l}...\partial x_{j_2}}\left(\frac{\partial f(x)}{\partial x_{j_1}}\right)\right)
\]

\noindent is zero modulo $p$. In particular, $B_l^{(p)} = 0$ for $l = 1, ... ,k-1$, and we have $\I_{g,1}^Z(k) < \I_{g,1}^P(k)$ as desired. When $l = 1$ we have have

\[
\epsilon\left(\frac{\partial^{l}f(x)x^{-1}}{\partial x_{j_l}...\partial x_{j_1}}\right) = \epsilon\left(\frac{\partial f(x)x^{-1}}{\partial x_{j_1}}\right)
\]

\noindent which we have seen before to be equal to:

\[
\begin{cases}
		\epsilon(\frac{\partial f(x)}{\partial x_{j_1}}) - 1, & \text{if } x = x_{j_1}\\
		\epsilon(\frac{\partial f(x)}{\partial x_{j_1}}), & \text{if } x \neq x_{j_1}
\end{cases}
\]

\noindent As we are under the assumption that this is equal to zero modulo $p$, we see that $B_0^{(p)} = Id$, and we have proven the theorem.

\qedhere
\end{proof}


\section{Calculating the images}\label{Section:Images}

Toward the goal of computing the images of the first Stallings and Zassenhaus mod-p Johnson homomorphisms we 
first describe a generating set for $Mod_{g,1}[p] = \I_{g,1}^S(1) = \I_{g,1}^Z(1)$. Let us introduce the notation 
$D_p$ for the collection of $p^{th}$-powers of Dehn twists. The following lemma was 
announced in \cite{Perron2} by Perron.

\begin{lem}\label{ModPGenerators}
$Mod_{g,1}[p]$ is generated by $\I_{g,1}$ and $D_p$.
\end{lem}

We would like to prove Lemma \ref{ModPGenerators} using the general procedure afforded by the following lemma,
whose proof is easy and omitted.

\begin{lem}\label{GeneratingGroups}
Consider the short exact sequence of groups:

\[
1 \longrightarrow A \longrightarrow G \stackrel{\rho}{\longrightarrow} B \longrightarrow 1
\]

\noindent and generating sets $S$ for $A$ and $S_B$ for $B$. Let $S_B' \subset G$ denote a set that projects onto $S_B$. The group $G$ is generated by $S_A \cup S_B'$.
\end{lem}

\noindent The short exact sequence we wish to use comes from the following short exact sequence:

\[
1 \longrightarrow \I_{g,1} \longrightarrow Mod_{g,1} \stackrel{\Psi}{\longrightarrow} Sp_{2g}(\Z) \longrightarrow 1
\]

Let $E_{i,j}$ be the $g\times g$ matrix whose entries are all zero save the $(i,j)$ and $(j,i)$ entries, which are $1$.
We define the two following symplectic matrices:

\[ M_{i,j} =  \left( \begin{array}{cc}
Id_{g} & pE_{i,j} \\
0 & Id_{g} \end{array} \right)\]

\[ N_{i,j} = \left( \begin{array}{cc}
Id_{g} & 0 \\
pE_{i,j} & Id_{g} \end{array} \right)\]

\begin{definition}
 $Sp_{2g}(\Z)[p] = ker(Sp_{2g}(\Z) \stackrel{\eta}{\longrightarrow} Sp_{2g}(\F_p))$, where $\eta$ is reduction
 of entries modulo $p$.
\end{definition}

\noindent $Sp_{2g}(\Z)[p]$ is called the level-$p$ congruence subgroup of $Sp_{2g}(\Z)$ and consists of matrices
congruent to the identity matrix modulo $p$.

\begin{thm}[Bass, Milnor, Serre, {\cite{Bass}}]\label{SpGenerators}
$Sp_{2g}(\Z)[p]$ is normally generated (as a subgroup of $Sp_{2g}(\Z)$) by the matrices $M_{i,j}$ and $N_{i,j}$.
\end{thm}

The matrices $M_{i,j}$ and $N_{i,j}$ have lifts to $Mod_{g,1}$ that are products of three $p^{th}$-powers of Dehn twists
in $Mod_{g,1}[p]$. Figure \ref{figure:LiftToTwists} shows the simple closed curves in the following lifts. 
Consider such a matrix and the usual symplectic basis $\{a_1, b_1,...,a_g, b_g\}$. 
The matrix does one of two things to our basis:

\begin{itemize}
\item $M_{i,j}$ sends $b_i\mapsto b_i + pa_j$ and $b_j\mapsto b_j + pa_i$ and fixes every other basis element
\item $N_{i,j}$ sends $a_i\mapsto a_i + pb_j$ and $a_j\mapsto a_j + pb_i$ and fixes every other basis element
\end{itemize}

The mapping class $T^{-p}_{a_i}T^{p}_{\delta}T^{-p}_{a_j}$ projects to $M_{i,j}$.
The mapping class $T^{-p}_{b_i}T^{p}_{\epsilon}T^{-p}_{b_j}$ projects to $N_{i,j}$. This can be checked easily on the basis elements $\{a_1, b_1,..., a_g, b_g\}$.

\begin{figure}[t]
\centering

\centerline{\psfig{file=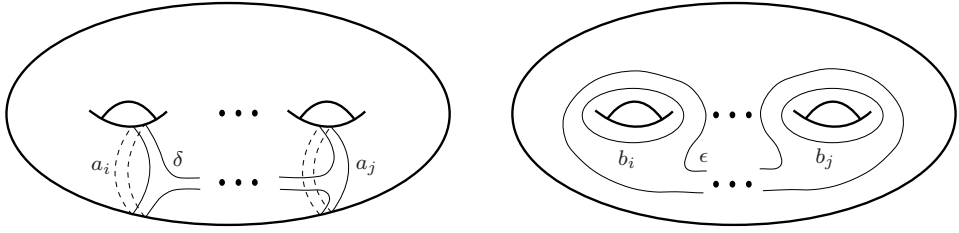,width=\textwidth}}
\caption{Curves involved in the lifts of $M_{i,j}$ and $N_{i,j}$}
\label{figure:LiftToTwists}
\end{figure}

In our proof of Lemma \ref{ModPGenerators} we will also need to make use of the following relation among Dehn twists. See \cite{Farb}, Fact 3.7, for instance.

\begin{lem}
\label{lemma:ConjTwists}
Let $T_a$ be a Dehn twist and $f$ be any other element of $Mod_{g,1}$ then $fT_af^{-1} = T_{f(a)}$.
\end{lem}

\begin{proof}[{Proof of Lemma \ref{ModPGenerators}}]
The symplectic representation of the mapping class group gives us the short exact sequence:

\[
1 \longrightarrow \I_{g,1} \longrightarrow Mod_{g,1} \stackrel{\Psi}{\longrightarrow} Sp_{2g}(\Z) \longrightarrow 1
\]

\noindent If we restrict $\Psi$ to the subgroup $Mod_{g,1}[p] < Mod_{g,1}$ we get the short exact sequence:

\[
1\longrightarrow \I_{g,1} \longrightarrow Mod_{g,1}[p] \stackrel{\Psi}{\longrightarrow} Sp_{2g}(\Z)[p] \longrightarrow 1
\]

\noindent From the latter exact sequence we may obtain a generating set for 
$Mod_{g,1}[p]$ by concatenating a generating set for $\I_{g,1}$ with the lift of a generating set for 
$Sp_{2g}(\Z)[p]$ to $Mod_{g,1}[p]$ by Lemma \ref{GeneratingGroups}.

$Sp_{2g}(\Z)[p]$ is normally generated (in $Sp_{2g}(\Z)$) by $p^{th}$-powers of the matrices $M_{i,j}$ and $N_{i,j}$ by Theorem \ref{SpGenerators}. Such matrices are the image of $p^{th}$-powers of Dehn twists under $\Psi$ (up to composition with elements of $\I_{g,1}$). Lemma \ref{lemma:ConjTwists} tells us that conjugates of $p^{th}$-powers of Dehn twists are still $p^{th}$-powers of Dehn twists. We consequently have a generating set for $Mod_{g,1}[p]$ that consists of $\I_{g,1}$ and all $p^{th}$-powers of Dehn twists, which we denote $D_p$.\qedhere
\end{proof}

The following lemma also follows from the above proof.

\begin{lem}
An element $f\in Mod_{g,1}[p]$ may be written as $f = f_1 \cdot f_2$ where $f_1\in \I_{g,1}$ and $f_2$ is a product of $p^{th}$-powers of Dehn twists.
\end{lem}

Now we may calculate the image of these generators under our homomorphisms $\tau_1^S$ and $\tau_1^Z$ 
to understand the images of these homomorphisms. For each quotient $\N_k^*$ we have a short exact sequence:

\[
1 \longrightarrow \G_{k+1}^* \longrightarrow \G \longrightarrow \N_k^* \longrightarrow 1
\]

\noindent and from this a 5-term exact sequence in homology.

\begin{align*}
H_2(\G; \F_p) &\longrightarrow H_2(\N_k^*; \F_p) \longrightarrow \frac{\G_{k+1}^*}{[\G,\G_{k+1}^*](\G_{k+1}^*)^p} \\
 \longrightarrow &H_1(\G; \F_p) \longrightarrow H_1(\N_k^*; \F_p) \longrightarrow 0
\end{align*}

\noindent Usually $H_1(\G_{k+1}^*;\F_p)_{{\N}_k^*}$ appears instead of $\frac{\G_{k+1}^*}{[\G,\G_{k+1}^*](\G_{k+1}^*)^p}$ in
such a 5-term exact sequence. In \cite{Stallings} Stallings shows that in this case these terms are isomorphic. 

\subsection{The Case of Stallings} Since $\G$ is a free group of rank $2g$ we have that $H_2(\G;\F_p)$ is trivial 
and $H_1(\G; \F_p) \cong H_1(\N_k^*; \F_p) \cong \F_p^{2g}$. For the case of Stallings 
$\frac{\G_{k+1}^S}{[\G,\G_{k+1}^S](\G_{k+1}^S)^p} \cong \el_{k+1}^S$, and our 5-term exact sequence reduces to:

\[
0 \longrightarrow H_2(\N_k^S; \F_p) \longrightarrow \el_{k+1}^S \longrightarrow \F_p^{2g} \stackrel{\cong}{\longrightarrow} \F_p^{2g} \longrightarrow 0
\]

\noindent so that $\el_{k+1}^S \cong H_2(\N_k^S; \F_p)$. In particular, we may consider the codomain of $\tau_1^S$ 
to be $Hom(\F_p^{2g}, H_2(\N_1^S; \F_p))$. Since $N_1^S \cong H_1(\G;\F_p) \cong \F_p^{2g}$ we may better understand
$H_2(\N_1^S; \F_p)$ by the split exact sequence given by the universal coefficient theorem:

\[
0 \longrightarrow H_2(\F_p^{2g})\otimes\F_p \longrightarrow H_2(\F_p^{2g}; \F_p) \longrightarrow Tor(H_1(\F_p^{2g}), \F_p) \longrightarrow 0
\]

\noindent Since $\F_p^{2g}$ is an abelian $p-$group we know that $H_2(\F_p^{2g})\cong\bigwedge^2(\F_p^{2g})$. See Chapter VI in \cite{Brown}, for instance.
By the unnatural splitting of the sequence we have that:

\[
H_2(\F_p^{2g}; \F_p) \cong (H_2(\F_p^{2g})\otimes\F_p)\oplus(Tor(H_1(\F_p^{2g}), \F_p)) \cong (\bigwedge^2\F_p^{2g})\oplus\F_p^{2g}
\]

\noindent and we may consider the image of $\tau_1^S$ to be contained in $Hom(\F_p^{2g}, (\bigwedge^2\F_p^{2g})\oplus\F_p^{2g})$.

\subsection{The Case of Zassenhaus} For Zassenhaus we have:

\[
\el_2^Z \neq \frac{\G_2^Z}{[\G,\G_2^Z](\G_2^Z)^p} = \frac{\G_2^S}{[\G,\G_2^S](\G_2^S)^p} = \el_2^S
\]

\noindent Rather $\el_2^Z$ is a quotient of $\el_2^S$. In particular, $\el_2^Z = \el_2^S/Q$ where $Q = \G_3^Z/\G_3^S$. We show that $Q \cong \F_p^{2g}$ and that $\el_2^Z$ contains $\bigwedge^2\F_p^{2g}$ as a subgroup. We conclude that $\el_2^Z \cong \bigwedge^2\F_p^{2g}$.
First recall that $\G_2^S = \G_2^Z = [\G,\G]\G^p$. 
Next note that for $p\geq 3$ we have $\G_3^S = [\G,[\G,\G]\G^p]([\G,\G]\G^p)^p < \G_3^Z = [\G,[\G,\G]]\G^p$. 

Now we show $Q \cong \F_p^{2g}$. Let $u\in [\G,[\G,\G]]$ and $v\in \G^p$. Both $u$ and $v$ are elements
of $\G_3^Z$. Note that the images of $u$ and $v$ commute in the quotient $Q$ since $Q$ is a subgroup of an abelian group. Since $[\G,[\G,\G]] < \G_3^S$ the image of $u$ in $Q$ is trivial, and $Q \cong \G^p/(\G_3^S\cap\G^p)$. Note that both $[\G^p,\G^p]$ and $\G^{p^2}$ are subgroups of $\G_3^S\cap\G^p$ so that $Q \cong \F_p^{2g}$.

It remains to find an injection $j:\bigwedge^2\F_p^{2g}\longrightarrow \el_2^Z$. Consider the inclusion $i:\bigwedge^2 \F_p^{2g} \longrightarrow \el_2^S$ defined by
$(a\wedge b) \mapsto ([a',b']\text{ mod }\G_3^S)$, where $a'$ and $b'$ are lifts of $a$ and $b$ respectively to $\G$.
Also consider the surjection $\rho:\el_2^S \longrightarrow \el_2^Z$, which is the quotient map by $Q$. We define $j:\bigwedge^2\F_p^{2g}\longrightarrow \el_2^Z$ by $j = \rho \circ i$. Note that $ker(j) \cong image(i) \cap ker(\rho)$, which is trivial since

\begin{align*}
image(i) &= \frac{[\G,\G]}{\G_3^S}\text{,}\\
ker(\rho) &= \frac{\G_3^Z}{\G_3^S}\text{,}
\end{align*}

\noindent and $[\G,\G]\cap\G_3^Z = [\G,[\G,\G]] < \G_3^S$. Therefore, $j$ is injective. We conclude that $\el_2^Z \cong \bigwedge^2\F_p^{2g}$ as desired.

In summary we have the following:
\begin{itemize}
\item $image(\tau_1^S) \subset Hom(\F_p^{2g}, (\bigwedge^2\F_p^{2g})\oplus\F_p^{2g})$
\item $image(\tau_1^Z) \subset Hom(\F_p^{2g}, \bigwedge^2\F_p^{2g})$
\end{itemize}

Now we note that:

\[
Hom(\F_p^{2g}, (\bigwedge^2\F_p^{2g})\oplus\F_p^{2g})) \cong Hom(\F_p^{2g}, \bigwedge^2\F_p^{2g}) \oplus Hom(\F_p^{2g}, \F_p^{2g}))
\]

\noindent The first factor contains $\bigwedge^3\F_p^n$ as a subgroup by the injection:

\[
x\wedge y \wedge z \mapsto x\otimes(y\wedge z) + y\otimes(z\wedge x) + z\otimes(x\wedge y)
\]

To understand the second factor we need the following definition. Let $\Omega$ be the skew-symmetric matrix \[ \left( \begin{array}{cc}
0 & Id_{g} \\
-Id_{g} & 0  \end{array} \right)\]

\begin{definition}
$\mathfrak{sp}_{2g}(\F_p)$ is the additive group of $2g \times 2g$ matrices $M$ with entries in $\F_p$ satisfying $M^T\Omega + \Omega M = 0$.
\end{definition}

It is independently a theorem of Perron \cite{Perron1}, Putman \cite{Putman2}, and Sato \cite{Sato} that this is the abelianization of 
$Sp_{2g}(\Z)[p]$ in our case when $p$ is an odd prime and $g \geq 2$. The abelianization map:

\[
abel:Sp_{2g}(\Z)[p]\longrightarrow \mathfrak{sp}_{2g}(\F_p)
\]

\noindent uses the observation that $X\in Sp_{2g}(\Z)[p]$ can be written
as $X = Id +pA$ for some matrix $A$ and is defined by:

\[
abel(X) = abel(Id + pA) = A\ \text{mod}\ p
\]

\noindent As a consequence of having an explicit abelianization map, one may conclude the following theorem.

\begin{thm}[Perron, {\cite{Perron1}}; Putman, \cite{Putman2}; Sato, \cite{Sato}]\label{ModPAbel}
For $g \geq 3$ and $p\geq 3$ odd, the abelianization $H_1(Sp_{2g}(\Z)[p])$ is isomorphic to $Sp_{2g}(\Z)[p]/Sp_{2g}(\Z)[p^2]$.
\end{thm}

\noindent We must also understand the action of a Dehn twist on homology before proceeding. Let $i(\cdot, \cdot)$ denote the algebraic intersection pairing on $H_1(\Sigma_{g,1})$. For a discussion of the following lemma see, for instance, \cite{Farb}, Proposition 6.3.

\begin{lem}\label{lemma:TwistonHomology}
Let $[y]\in H_1(\Sigma_{g,1})$ be a primitive element with unbased simple closed curve representative $\gamma$. The $k^{th}$-power of a Dehn twist $T_{\gamma}^k$ about $\gamma$ induces the map on homology

\begin{align*}
[T_{\gamma}^k]:H_1(\Sigma_{g,1})&\longrightarrow H_1(\Sigma_{g,1})\\
[x]&\longmapsto [x] + k\cdot i([x],[y])[y]
\end{align*}
\end{lem}

The following lemma is also needed to understand the image of $p^{th}$-powers of Dehn twists.

\begin{lem}\label{lemma:TwistContributions}
Let $\gamma \subset \Sigma_{g,1}$ be a simple closed curve and $y\in\G$ be a based curve isotopic to $\gamma$. Let $x\in\G$ with $|x\cap y| = n$. Then

\[
T_{\gamma}^p(x)x^{-1} = \prod\limits_{i = 1}^{n}z_iy^{\pm p}z_i^{-1}
\]

\noindent for some $z_i \in \G$.
\end{lem}

\begin{proof}
Let $p_1,..., p_n$ denote the intersection points of $x$ and $y$. Let $p_0$ denote the basepoint of $\G$. Let $a_i$ denote the the subarc of $x$ that goes from $p_0$ to $p_i$. Let $b_i$ denote the subarc of $y$ that goes from $p_0$ to $p_i$. One can easily see by direct calculation that  $T_{\gamma}^p(x) = \left(\prod\limits_{i=1}^{n}a_ib_i^{-1}y^{\epsilon(p_i)p}b_ia_i^{-1}\right)x$, where $\epsilon(p_i) = \pm 1$ depending on the sign of the intersection between $x$ and $y$ at $p_i$. 
\end{proof}

\begin{thm}
$image(\tau_1^Z) \cong \bigwedge^3\F_p^{2g}$ for $g\geq 2$.
\end{thm}

\begin{proof}
We begin by considering $f\in\I_{g,1}$ and $[x]\in H_1(\Sigma_{g,1};\F_p)$, the class of the simple closed curve $x$.
We claim that $\tau_1^Z(f)([x]) = \tau_1(f)([x])\ \text{mod}\ p$. 
First observe that $f(x)x^{-1} \in [\G, \G] = \G_2 < \G_2^Z$. This allows the following equalities:

\begin{align*}
\tau_1^Z(f)([x]) &= f(x)x^{-1}\ \text{mod}\ \G_3^Z \\
	&= f(x)x^{-1}\ \text{mod}\ [\G, [\G, \G]]\G^p\text{, by definition of }\G_3^Z \\
	&= f(x)x^{-1}\ \text{mod}\ \G_3\G^p\text{, by definition of }\G_3\\
	&= \tau_1(f)([x])\ \text{mod}\ p 
\end{align*}

\noindent Johnson in \cite{Johnson0} proves $\tau_1(\I_{g,1}) \cong \bigwedge^3 \Z^{2g}$. It follows that $\tau_1^Z(\I_{g,1}) \cong \bigwedge^3\F_p^{2g}$.

Now let us consider the case when $f = T_{\gamma}^p \in D_p$. Let $y$ be a based curve in the isotopy class of $\gamma$ as in Lemma \ref{lemma:TwistContributions}. Since $[\G,\G^p]<\G_3^Z$ we have that $z_iy^{\pm p}z_i^{-1} \equiv y^{\pm p} \text{ mod }\G_3^Z$ for any $z_i \in \G$. This fact and Lemma \ref{lemma:TwistContributions} let us conclude that

\[
f(x)x^{-1} = \prod\limits_{i = 1}^{n}z_iy^{\pm p}z_i^{-1} \equiv y^{p\cdot k}\text{ mod }\G_3^Z
\]

\noindent where $n = |x\cap y|$ and $k$ is the sum of the exponents $\pm p$ of $y$. Next notice that for $p\geq 3$ we have $\G^p < \G_3^Z$, so in fact $f(x)x^{-1} \equiv 0\ \text{mod}\ \G_3^Z$. Hence, $image(\tau_1^Z) = \tau_1^Z(\I_{g,1}) \cong \bigwedge^3\F_p^{2g}$.\qedhere
\end{proof}

\begin{thm}
For $g\geq 2$ there is an unnaturally split short exact sequence:

\[
0\longrightarrow \bigwedge^3\F_p^{2g} \longrightarrow image(\tau_1^S) \longrightarrow \mathfrak{sp}_{2g}(\F_p) \longrightarrow 0
\]
\end{thm}

\begin{proof}
First we show that $\tau_1^S(\I_{g,1}) \cong \tau_1^Z(\I_{g,1})$. Consider the map:

\begin{align*}
j:\tau_1^S(\I_{g,1}) &\longrightarrow \tau_1^Z(\I_{g,1}) \\
 f(x)x^{-1}\text{ mod }\G_3^S &\longmapsto f(x)x^{-1}\text{ mod }\G_3^Z
\end{align*}

\noindent where $f\in\I_{g,1}$. This map is surjective since it is given by taking the quotient by $Q = \G_3^Z/\G_3^S$. To see that $j$ is injective consider an element of $ker(j)$. This element is given by some $f\in\I_{g,1}$ such that $f(x)x^{-1}\text{ mod }\G_3^S \in Q$ for all $x\in\G$. We know that $f(x)x^{-1}$ is always in $\G_2$ since $f\in \I{g,1}$. It must then be the case that $f(x)^{-1}\text{ mod }\G_3^S \in image(i)$, where again $i$ is the inclusion $\bigwedge^2\F_p^{2g}\hookrightarrow \el_2^S$ defined earlier in this section. We also saw earlier in this section that $image(i) \cap Q$ is trivial. Therefore $j$ is injective and hence an isomorphism. It follows that $\tau_1^S(\I_{g,1}) \cong \bigwedge^3\F_p^{2g} < Hom(\F_p^{2g}, \bigwedge^2 \F_p^{2g})$.

Now recall that we have the unnaturally split short exact sequence:

\[
1\longrightarrow Hom(\F_p^{2g}, \bigwedge^2 \F_p^{2g}) \longrightarrow Hom(\F_p^{2g},\el_2^S) \stackrel{r}{\longrightarrow} Hom(\F_p^{2g},\F_p^{2g}) \longrightarrow 1
\]

\noindent Below we construct an injective composition of maps

\[
\mathfrak{sp}_{2g}(\F_p) \stackrel{\hat{\tau}}{\longrightarrow} \frac{Hom(\F_p^{2g},\el_2^S)}{\bigwedge^3 \F_p^{2g}} \stackrel{\overline{r}}{\longrightarrow} Hom(\F_p^{2g},\F_p^{2g})
\]

\noindent where $\overline{r}$ is induced by the map $r$ in the short exact sequence above. Again $\bigwedge^3\F_p^{2g} < Hom(\F_p^{2g}, \bigwedge^2 \F_p^{2g})$ is the image of $\I_{g,1}$ under $\tau_1^S$. This yields the desired short exact sequence of the theorem.

The calculation above that $\tau_1^S(\I_{g,1}) \cong \bigwedge^3 \F_p^{2g}$ lets us first define the map

\[
\hat{\tau}':\frac{\I_{g,1}^S(1)}{\I_{g,1}}\longrightarrow \frac{Hom(\F_p^{2g},\el_2^S)}{\bigwedge^3 \F_p^{2g}},
\]

\noindent induced by $\tau_1^S$. Recall that $\I_{g,1}^S(1)/\I_{g,1} \cong Sp_{2g}(\Z)[p]$. Also recall that $Sp_{2g}(\Z)[p^2]$ normally generated by matrices that lift to products of $p^2$-powers of Dehn twists in $Mod_{g,1}$. Next note that $p^2$-powers of Dehn twists are in $\I_{g,1}^S(2)$ since $\G^{p^2} < \G_3^S$. Our map $\hat{\tau}'$ then induces the map:

\[
\hat{\tau}:\frac{Sp_{2g}(\Z)[p]}{Sp_{2g}(\Z)[p^2]}\cong \mathfrak{sp}_{2g}(\F_p)\longrightarrow \frac{Hom(\F_p^{2g},\el_2^S)}{\bigwedge^3 \F_p^{2g}}
\]

Before showing that $\overline{r}\circ\hat{\tau}$ is injective we show that $r\circ\tau_1^S = abel \circ \Psi$ on $p^{th}$-powers of Dehn twists. Again $\Psi:Mod_{g,1}\longrightarrow Sp_{2g}(\Z)$ is the usual symplectic representation, and $abel:Sp_{2g}(\Z)[p]\longrightarrow\mathfrak{sp}_{2g}(\F_p)$ is the abelianization map.

Consider the case when $f = T_{\gamma}^p \in D_p$. So that $f \notin \I_{g,1}$ 
we will only consider the case when $\gamma$ is nonseparating. Let $y$ be a based representative of $\gamma$. As in the Zassenhaus calculation, we again have that $f(x)x^{-1} \equiv y^{p\cdot k}\ \text{mod}\ \G_3^S$ for some $k$ since Lemma \ref{lemma:TwistContributions} still applies, and $[\G,\G^p]<\G_3^S$ so that conjugation is trivial in the quotient $\el_2^S$. However, it is now the case that $\G^p \not< \G_3^S$ when $p\geq 3$, so $f(x)x^{-1}$ is not necessarily trivial modulo $\G_3^S$. As in the case for Zassenhaus, we appeal to Lemma \ref{lemma:TwistContributions} to write:

\[
T_{\gamma}^p(x)x^{-1} = \prod\limits_{i = 1}^{n}z_iy^{\pm p}z_i^{-1} \equiv  y^{p\cdot k}\ \text{mod}\ \G_3^S
\]

\noindent Where $k$ is again the sum of the exponents $\pm p$ of $y$. We saw in the proof of Lemma \ref{lemma:TwistContributions} that the sign of each exponent is determined by the sign of the corresponding intersection. We conclude that $k$ is the algebraic intersection number of $[x]$ with $[y]$, denoted $i([x],[y])$. Now we have:

\[
\tau_1^S(f)([x]) \equiv  y^{p\cdot i([x],[y])}\ \text{mod}\ \G_3^S
\]

\noindent and

\[
r\circ\tau_1^S(f)([x]) \cong i([x],[y])[y] \text{ mod }p \in H_1(\Sigma_{g,1};\F_p) \cong \F_p^{2g}
\]

Recall by Lemma \ref{lemma:TwistonHomology}:

\[
\Psi(f)([x]) =  [x] + p\cdot i([x],[y])[y]
\]

\noindent It follows that

\[
abel\circ\Psi(f)([x]) = i([x],[y])[y] \text{ mod }p
\]

\noindent Therefore, $abel\circ\Psi(f) = r\circ\tau_1^S(f)$ when $f$ is a product of $p^{th}$-powers of Dehn twists.

Now that we have seen $r\circ\tau_1^S = abel \circ \Psi$ on $p^{th}$-powers of Dehn twists, we may show $\overline{r}\circ\hat{\tau}$ to be injective. We have the following diagram:

\[
\begin{CD}
\I_{g,1}^S(1)   @>\text{$\Psi$}>> \I_{g,1}^S(1)/\I_{g,1}   @>\text{abel}>>  \mathfrak{sp}_{2g}(\F_p)\\
@V\text{$\tau_1^S$}VV   @V\text{$\hat{\tau}'$}VV     @V\text{$\hat{\tau}$}VV             \\
Hom(\F_p^{2g},\el_2^S)  @>>> Hom(\F_p^{2g},\el_2^S)/\bigwedge^3\F_p^{2g}      @>\text{id}>>  Hom(\F_p^{2g},\el_2^S)/\bigwedge^3\F_p^{2g} \\
@V\text{r}VV    @V\text{$\overline{r}$}VV    @V\text{$\overline{r}$}VV \\
Hom(\F_p^{2g},\F_p^{2g}) @>\text{id}>> Hom(\F_p^{2g},\F_p^{2g}) @>\text{id}>> Hom(\F_p^{2g},\F_p^{2g})
\end{CD}
\]

\noindent We use the outermost square of the above diagram to argue the desired injectivity. Consider $\phi\in ker(\overline{r}\circ\hat{\tau})$. We may find a mapping class $f\in\I_{g,1}^S(1)$ such that $abel\circ\Psi(f) = \phi$. We have seen that $r\circ\tau_1^S(f) = \phi$ as well. Since the bottom map of the outer square is the identity map, $\phi$ itself must be trivial.\qedhere
\end{proof}

We emphasize that the above work is packaged in the commutativity of the following diagram.

\[
\begin{CD}
\I_{g,1}               @>>> Mod_{g,1}[p]          @>\text{$\Psi$}>>  Sp_{2g}(\Z)[p]           @>>> 1\\
@V\text{$\tau_1 \text{ mod } p$}VV     @V\text{$\tau_1^S$}VV                    @V\text{abel}VV    \\
\bigwedge^3 \F_p^{2g}  @>>> image(\tau_1^S)       @>>>               \mathfrak{sp}_{2g}(\F_p) @>>> 0
\end{CD}
\]


\section{Morita's image restriction}\label{Section:Morita}
In \cite{Morita} Morita describes a subgroup of $Hom(H, \el_{k+1})$ still containing the image of $\tau_k$. Using a similar argument we find subgroups of $Hom(H_1(\Sigma_{g,1};\F_p), \el_{k+1}^*)$ filling an analogous role for our mod-p Johnson homomorphisms.

We begin by outlining the work of Morita that generalizes completely. Consider a 2-chain $\sigma$ of our 
free group $\G$ such that $\partial \sigma = -z$, where $z$ represents the boundary in $\Sigma_{g, 1}$. Let $f \in Mod_{g,1}$ and define
$\sigma_f = \sigma - f_*(\sigma)$. Since a mapping class $f$ preserves the boundary pointwise, $\sigma_f$ is a 
2-cycle. Furthermore, since $H_2(\G) = 0$ there exists a 3-chain $c_f$ such that $\partial c_f = \sigma_f$. 
For a 3-cycle $c$ in $\G$ let $c'(k) = c\ \text{mod}\ \G_{k+1}^*$ be the image of $c$ in the chains of $\N_k^*$. 
Now suppose $f \in \I_{g,1}^*(k)$, one of the Zassenhaus, Stallings, or usual Johnson filtration terms. In
this case $\partial c_f'(k) = \sigma'(k) - f(\sigma'(k)) = 0$ since $f$ acts trivially on $\N_k^*$ by the 
definitions of $\I_{g,1}^*(k)$. Hence, $c_f'(k)$ is a representative of an element of $H_3(\N_k^*)$. We may 
then define a map ${\tau_k^*}'(f) = [c_f'(k)]$ so that ${\tau_k^*}':\I_{g, 1}^*(k) \rightarrow H_3(\N_k^*)$.
By Theorem \ref{MoritaModP} below ${\tau_k^*}'$ is a well-defined homomorphism.

Consider the short exact sequence:

\[
0 \longrightarrow \el_{k+1}^* \longrightarrow \N_{k+1}^* \longrightarrow \N_k^* \longrightarrow 1
\]

\noindent Associated to such an exact sequence is the Lyndon-Hochschild-Serre spectral sequence $H_p(\N_k^*; H_q(\el_{k+1}^*;R)) \Rightarrow H_{p+q}(\N_{k+1}^*;R)$, where $R$ is $\Z$ if $*$ is
empty and $\F_p$ otherwise. See \cite{Brown}, Section VII, for instance. Let us denote the differentials 
$d_{p,q}^r:E_{p,q}^r\longrightarrow E_{p-2,q+r-1}^r$. The following is due to Morita.

\begin{thm}[Morita, {\cite{Morita}, Theorem 3.1}]
The map $\tau_k':\I_{g, 1}(k) \rightarrow H_3(\N_k)$ is well-defined and satisfies the commutativity condition $\tau_k = d \circ \tau_k'$.
\end{thm}

\noindent Morita's proof also works to prove the analogous result for our mod-p cases.

\begin{thm}\label{MoritaModP}
The map ${\tau_k^*}':\I_{g, 1}^*(k) \rightarrow H_3(\N_k^*;\F_p)$ is well-defined and satisfies the commutativity condition $\tau_k^* = d \circ {\tau_k^*}'$.
\end{thm}

\noindent Note that we have a short exact sequence following from the definition of the third page term 
$E_{1,1}^3 = ker(d_{1,1}^2)/image(d_{3,0}^2)$:

\[
1 \longrightarrow E_{3,0}^2 \stackrel{d_{3,0}^2}{\longrightarrow} E_{1,1}^2 \stackrel{\rho}{\longrightarrow} E_{1,1}^3 \longrightarrow 1
\]

\noindent This is isomorphic to the short exact sequence:

\[
0 \longrightarrow H_3(\N_k^*;R) \stackrel{d_{3,0}^2}{\longrightarrow} \el_{k+1}^*\otimes H_1(\N_k^*;R) \stackrel{\rho}{\longrightarrow} \el_{k+2}^* \longrightarrow 0
\]

\noindent It follows from exactness that

\begin{align*}
image(\tau_k) &= image(d_{3,0}^2 \circ \tau_k')\\
&\subset image(d_{3,0}^2)\\
&= ker(\rho) \\
&\cong E_{3,0}^2 \\
&= H_3(\N_k) \\
\end{align*}

\noindent Hence, we obtain the following result of Morita:

\begin{cor}
$image(\tau_k) < H_3(\N_k)$.
\end{cor}

\noindent Morita's proof also works to prove the following corollary to Theorem \ref{MoritaModP} 
with the one additional observation that $E_{3,0}^2 = H_3(\N_k^*;\F_p)$ in the mod-p cases.

\begin{cor}
$image(\tau_k^*) < H_3(\N_k^*;\F_p)$.
\end{cor}

We now describe how $H_3(\N_k^*;\F_p)$ is a subgroup of $Hom(H_1(\Sigma_{g,1};\F_p), \el_{k+1}^*)$ for the two cases 
of Stallings and Zassenhaus. Since $\F_p$ is a field the spectral sequences in each case converge with 
$H_2(\N_{k+1}^*;\F_p) \cong E_{2,0}^{\infty} \oplus E_{1,1}^{\infty} \oplus E_{0,2}^{\infty}$. Further, since 
we are dealing with a first quadrant spectral sequence each of these summands are stabilized by the third page. 
Consequently: 

\[
H_3(\N_k^*;\F_p)\cong ker(\rho) = ker\left(Hom(H_1(\Sigma_{g,1};\F_p), \el_{k+1}^*)\longrightarrow \frac{H_2(\N_{k+1}^*;\F_p)}{E_{2,0}^3\oplus E_{0,2}^3}\right)
\]

\noindent Morita shows that $E_{2,0}^3$ and $E_{0,2}^3$ are trivial for the usual lower central series. For our series this is not the case. We now aim to describe the term $\frac{H_2(\N_{k+1}^*;\F_p)}{E_{2,0}^3\oplus E_{0,2}^3}$ for both Stallings and Zassenhaus. By definition we have:
\begin{itemize}
\item $E_{2,0}^3 = \frac{ker(d_{2,0}^2)}{image(d_{4,-1}^2)}$
\item $E_{0,2}^3 = \frac{ker(d_{0,2}^2)}{image(d_{2,1}^2)}$
\end{itemize}

\subsection{The Case of Stallings} First, we show that the following differential is an isomorphism:

\[
d_{2,0}^2:H_2(\N_k^S;\F_p) \rightarrow H_1(\el_{k+1}^S;\F_p)
\]

\noindent This lets us conclude that $E_{2,0}^3$ is trivial. We begin by considering the commutative
diagram with short exact rows:

\[
\begin{CD}
1 @>>> \G_{k+1}^S @>>> \G @>>> \N_k^S @>>> 1\\
@.       @VVV    @VVV     @|     @.\\
0 @>>> \el_{k+1}^S @>>> \N_{k+1}^S @>>> \N_k^S @>>> 1
\end{CD}
\]

\noindent To each short exact sequence above is associated a long exact sequence in homology. Each long exact
sequence reduces as in Section \ref{Section:Images} to yield the following commutative diagram:

\[
\begin{CD}
H_2(\G;\F_p)    @>>> H_2(\N_k^S;\F_p) @>\text{g}>> \el_{k+1}^S @>>> 0 \\
@VVV                     @|                           @|         @.\\
H_2(\N_{k+1}^S;\F_p) @>>> H_2(\N_k^S;\F_p) @>\text{d}>> \el_{k+1}^S @>>> 0
\end{CD}
\]

\noindent We note that $H_2(\G;\F_p) \cong 0$ since $\G$ is free. Hence $g$ is an isomorphism. By commutativity we conclude that $d_{2,0}^2$ is also an isomorphism.

Now we want to understand the differential:

\[
d_{2,1}^2:H_2(\N_k^S)\otimes \el_{k+1}^S \rightarrow H_2(\el_{k+1}^S;\F_p).
\]
 
\noindent Note that:

\[
H_2(\N_k^S)\otimes \el_{k+1}^S\cong H_2(\N_k^S)\otimes \F_p \otimes \el_{k+1}^S.
\]
 
\noindent We saw in the previous section that $H_2(\N_k^S;\F_p) \cong \el_{k+1}^S$, and by the universal coefficient 
theorem $H_2(\N_k^S;\F_p) \cong H_2(\N_k^S)\otimes \F_p \oplus Tor(H_1(\N_k^S),\F_p)$. 
Hence,

\[
d_{2,1}^2:\frac{\el_{k+1}^S}{Tor(H_1(\N_k^S),\F_p)}\otimes \el_{k+1}^S \rightarrow H_2(\el_{k+1}^S;\F_p).
\]

By the universal coefficient theorem:

\[
H_2(\el_{k+1}^S;\F_p) \cong H_2(\el_{k+1}^S)\otimes\F_p\oplus Tor(\el_{k+1}^S, \F_p),
\]
 
\noindent and by virtue of $\el_{k+1}^S$ being a $\F_p$-vector space we finally we have that our differential is of the form:

\[ 
d_{2,1}^2:\frac{\el_{k+1}^S}{Tor(H_1(\N_k^S),\F_p)}\otimes \el_{k+1}^S \longrightarrow \bigwedge^2\el_{k+1}^S \oplus \el_{k+1}^S
\]

\subsection{The Case of Zassenhaus} For Zassenhaus we still have that:

\[
H_2(\N_k^Z;\F_p) \cong \frac{\G_k^Z}{[\G,\G_k^Z](\G_k^Z)^p}.
\]

\noindent And again, $H_2(\el_k^Z;\F_p) \cong \bigwedge^2\el_k^Z \oplus \el_k^Z$ by the universal coefficient theorem since $\el_k^Z$ is also a $\F_p$-vector space. Hence, we may consider the map:

\[
d_{2,1}^2:\frac{\G_k^Z}{[\G,\G_k^Z](\G_k^Z)^p} \otimes \el_k^Z \longrightarrow \bigwedge^2\el_k^Z \oplus \el_k^Z
\]

\noindent However, we do not have that $d = d_{2,0}^2$ is an injection in the Zassenhaus case. This is because we have the following commutative diagram:

\[
\begin{CD}
H_2(\G;\F_p)    @>>> H_2(\N_k^Z;\F_p) @>\text{g}>> \frac{\G_{k+1}^Z}{[\G,\G_{k+1}^Z](\G_{k+1}^Z)^p} @>>> 0\\
@VVV                     @|                           @VVV         @.\\
H_2(\N_{k+1}^Z;\F_p) @>>> H_2(\N_k^Z;\F_p) @>\text{d}>> \el_{k+1}^Z @>>> 0
\end{CD}
\]

\noindent If $d$ were injective it would also be an isomorphism by the exactness of the bottom row. 
Commutativity would then imply since $g$ is also an isomorphism that 
$\frac{\G_{k+1}^Z}{[\G,\G_{k+1}^Z](\G_{k+1}^Z)^p} \cong \el_{k+1}^Z$.

\subsection{Collecting the Results}

We therefore have the following statements concerning $H_3(\N_{k+1}^*;\F_p)$ as subgroups of the range of our mod-p Johnson homomorphisms:

\begin{itemize}
\item $H_3(\N_{k+1}^S;\F_p) \cong ker\left(Hom(H_1(\Sigma_{g,1};\F_p),\el_{k+1}^S)\longrightarrow \frac{\el_{k+2}^S}{E_S}\right)$
\item $H_3(\N_{k+1}^Z;\F_p) \cong ker\left(Hom(H_1(\Sigma_{g,1};\F_p),\el_{k+1}^Z)\longrightarrow \frac{\frac{\G_{k+1}^Z}{[\G,\G_{k+1}^Z](\G_{k+1}^Z)^p}}{E_Z}\right)$
\end{itemize}
where

\[
E_S = \frac{\bigwedge^2\el_{k+1}^S\oplus\el_{k+1}^S}{image\left(\frac{\el_{k+1}^S}{Tor(H_1(\N_k^S),\F_p)}\otimes \el_{k+1}^S \longrightarrow \bigwedge^2\el_{k+1}^S \oplus \el_{k+1}^S\right)}
\]

\noindent and

\begin{align*}
E_Z &= ker(H_2(\N_k^Z;\F_p)\longrightarrow\el_{k+1}^Z) \\
&\ \ \ \ \ \oplus \frac{\bigwedge^2\el_{k+1}^Z\oplus\el_{k+1}^Z}{image\left(\frac{\G_{k+1}^Z}{[\G,\G_{k+1}^Z](\G_{k+1}^Z)^p} \otimes \el_{k+1}^Z \longrightarrow \bigwedge^2\el_{k+1}^Z \oplus \el_{k+1}^Z\right)}
\end{align*}


\section{Generating the kernels}\label{Section:Kernels}
In this section we describe generating sets for the kernels of $\tau_1^*$.
We begin by giving a general procedure for obtaining a generating set for the kernel of a surjection.

\begin{lem}\label{lemma:GeneratingKernels}
Consider the short exact sequence of groups:

\[
1 \longrightarrow A \longrightarrow G \stackrel{\rho}{\longrightarrow} B \longrightarrow 1
\]

\noindent and a presentation $\left\langle S_G | R_G \right\rangle$ for $G$. Let $\left\langle S_B | R_B \right\rangle$ be a presentation for $B$ with $S_B = \rho(S_G)$. Write each $r\in R_B$ as a word $\prod\limits_i \rho(s_i)$ in $S_B$. For each $r$ define $w_r = \prod\limits_i s_i$ in $S_G$. The group $A$ is normally generated by $R_B' = \{w_r | r\in R_B\}$.
\end{lem}

\begin{proof}
By exactness of our sequence we know that $A = ker(\rho)$. Let $g \in G$ be such that $\rho(g) = 1$. Write $g$ as a word $w_g \in \left\langle S_G \right\rangle$. Then $\rho(w_g)$ is a word $w$ in $\left\langle S_B \right\rangle$. Since $g \in ker(\rho)$ we necessarily have that $w \in \left\langle\left\langle R_B \right\rangle\right\rangle$. We may therefore lift $w$ to obtain a word $w' \in \left\langle \left\langle \rho^{-1}(R_B) \right\rangle \right\rangle = \left\langle \left\langle R_B' \right\rangle \right\rangle$. So we have found a word $w'$ that is a product of conjugates of $R'$ elements and such that $w' = g$. We may therefore generate $A$ by taking a $\rho$-lift of each element of $R_B$ and taking the normal closure of the subgroup of $G$ generated by this collection of elements.\qedhere
\end{proof}

Recall that Lemma \ref{ModPGenerators} gives us a generating set for $Mod_{g,1}[p]$. In order to get
nice generating sets for $ker(\tau_1^*)$ we refine this
generating set for $Mod_{g,1}[p]$. We do this by considering a generating set for $\I_{g,1}$, rather than taking 
it in its entirety. Let $f_{a,b}$ denote the mapping class $T_aT_b^{-1}$, where $a$ and $b$ are disjoint and homologous simple closed curves. We call $f_{a,b}$ a bounding pair map. Recall a mapping class $T_{\gamma}$ is a separating twist when $\gamma$ is a 
separating simple closed curve in $\Sigma_{g,1}$.

\begin{thm}[Powell, {\cite{Powell}}]
$\I_{g,1}$ is generated by separating twists and bounding pair maps.
\end{thm}

Also recall that $ker(\tau_1)$ is generated by separating twists, which we denote $S'$. We may generate $\I_{g,1}$ by the set $S''$ containing $S'$ and a collection of bounding pair maps whose image under $\tau_1$ generates $image(\tau_1) \cong \bigwedge^3\Z^{2g}$. Johnson provides the following lemma as a corollary to \cite{Johnson2}, Theorem 4A.

\begin{lem}[Johnson, \cite{Johnson2}]\label{GenImageTorelli}
For $g\geq 3$, $\I_{g,1}/\K_{g,1}$ is generated by $2g \choose 3$ bounding pair maps.
\end{lem}

So by the above lemma we may choose $S''$ to contain $S'$ and only $2g \choose 3$ bounding pair maps. We may then obtain a generating set $S$ for $Mod_{g,1}[p]$ that contains $S''$ and the set of $p^{th}$-powers of Dehn twists, which we denote by $D_p$. We have calculated the image of $S$ under $\tau_1^*$ in Section \ref{Section:Images}, and so we may use the general argument outlined by Lemma \ref{lemma:GeneratingKernels} to obtain a normal generating set for the subgroups $\I_{g,1}^*(2) < Mod_{g,1}[p]$. In particular, we apply the procedure to the following short exact sequence for $* = S$ or $Z$:

\[
1 \longrightarrow \I_{g,1}^*(2) \longrightarrow Mod_{g,1}[p] \longrightarrow image(\tau_1^*) \longrightarrow 1
\]

The presentation for $image(\tau_1^*)$ uses the generating set $\tau_1(S)$ with some relators $R$. We must determine these relators and choose a lift of each to a word in $S$. In so doing we obtain a normal generating set for $\I_{g,1}^*(2)$. The normal generating sets are actually honest generating sets. 
In order to see this recall Lemma \ref{lemma:ConjTwists}, which states that conjugates of Dehn twists in $Mod_{g,1}$ are Dehn twists.

\begin{thm}
For $g\geq 3$, $\I_{g,1}^Z(2) = ker(\tau_1^Z)$ is generated by the group $\K = ker(\tau_1)$, which is generated by separating twists, and $D_p$, which is generated by $p^{th}$-powers of Dehn twists.
\end{thm}

\begin{proof}
We have the short exact sequence:

\[
1 \rightarrow \ker(\tau_1^Z) \rightarrow Mod_{g,1}[p] \rightarrow \bigwedge^3(\F_p^{2g}) \rightarrow 1
\]

\noindent Let $S$ be the generating set for $Mod_{g,1}[p]$ discussed above. That is, $S$ contains separating twists, $p^{th}$ powers of Dehn twists, and $2g \choose 3$ bounding pair maps. We may obtain a group presentation for $\bigwedge^3(\F_p^{2g})$ with generating set $\tau_1^Z(S)$ and relations given below. The relators in this presentation lift to 
words in $S \subset Mod_{g,1}[p]$, which give a normal generating set for $ker(\tau_1^Z)$ by 
Lemma \ref{lemma:GeneratingKernels}.

We now list the relations in our presentation for $\bigwedge^3(\F_p^{2g})$. We have seen that $\tau_1^Z(D_p) = 0$ so take each generator of the form $\tau_1^Z(T_{\gamma}^p)$ to be a relator. Recall $\tau_1^*|_{{\I}_{g,1}} = \tau_1\text{ mod }p$. The image of the separating twists are trivial since $\tau_1^Z(\K_{g,1}) = \tau_1(\K_{g,1})\text{ mod }p = 0$. So take each generator of the form $\tau_1^Z(T_c)$, where $c$ is separating, to be a relator.

Now we need only consider the relations among the images of the bounding pair maps $f_{a,b} = T_aT_b^{-1}$ in $S$. Since $\tau_1^Z(f_{a,b}) = \tau_1(f_{a,b})\text{ mod }p$, Lemma \ref{GenImageTorelli} tells us that the collection of $\tau_1^Z(f_{a,b})$ is a generating set for $\bigwedge^3 \F_p^{2g}$ with $2g \choose 3$ elements. Because $\bigwedge^3 \F_p^{2g} \cong \F_p^{2g\choose 3}$ we have the following relations among generators of the form $\tau_1^Z(f_{a,b})$:
\begin{itemize}
\item $[\tau_1^Z(f_{a,b}),\tau_1^Z(f_{d,e})]$
\item $\tau_1^Z(f_{a,b})^p$
\end{itemize}
Our complete list of relators in our presentation for $\bigwedge^3 \F_p^{2g}$ with generating set $\tau_1^Z(S)$ is as follows:
\begin{itemize}
\item $\tau_1^Z(T_{\gamma}^p)$
\item $\tau_1^Z(T_c)$, where $c$ is separating
\item $[\tau_1^Z(f_{a,b}),\tau_1^Z(f_{d,e})]$
\item $\tau_1^Z(f_{a,b})^p$
\end{itemize}
\noindent The relators $\tau_1^Z(T_{\gamma}^p)$ lift to $p^{th}$-powers of Dehn twists. The relators $\tau_1^Z(T_c)$ lift to separating twists. The relators $[\tau_1^Z(f_{a,b}),\tau_1^Z(f_{d,e})]$ lift to elements of $\K_{g,1}$, which is generated by separating twists. The relators $\tau_1^Z(f_{a,b})^p$ lift to $p^{th}$-powers of the bounding pair maps $f_{a,b}$. Note that $(T_aT_b^{-1})^p = T_a^pT_b^{-p}$, which is a product of $p^{th}$-powers of Dehn twists, since $a$ and $b$ are disjoint so that $T_a$ and $T_b$ commute. Therefore, $\I_{g,1}^Z(2)$ is normally generated by separating twists and $p^{th}$-powers of Dehn twists.

We note that the identity of Lemma \ref{lemma:ConjTwists} $fT_c^pf^{-1} = T_{f(c)}^p$ shows $\left< D_p\right>$ to be normal in $Mod_{g,1}[p]$. Similarly a conjugate of a separating twist is a separating twist. It follows that in fact $ker(\tau_1^Z) = \I_{g,1}^Z(2)$ is generated by separating twists and $p^{th}$-powers of Dehn twists.\qedhere
\end{proof}

\begin{thm}
For $g\geq 3$, $\I_{g,1}^S(2) = ker(\tau_1^S)$ is generated by separating twists, $p^{th}$-powers of bounding pair maps, and $p^2$-powers of Dehn twists.
\end{thm}

\begin{proof}
This proof is similar to the previous one for the Zassenhaus case except that we no longer have 
$D_p < \I_{g,1}^S(2)$. Again we take $S$ as defined above to be our generating set for $Mod_{g,1}[p]$ and $\tau_1^S(S)$ to be the generating set for a presentation of $image(\tau_1^S) \cong \bigwedge^3\F_p^{2g}\oplus \mathfrak{sp}_{2g}(\F_p)$. For the same reasons as in the case of Zassenhaus the generators $\tau_1^S(T_c)$ are relators when $c$ is separating. Again the images of the bounding pair maps $f_{a,b}$ generate the $\bigwedge^3\F_p^{2g}$ summand of $image(\tau_1^S)$, and we have the following relations among them:
\begin{itemize}
\item $[\tau_1^S(f_{a,b}),\tau_1^S(f_{d,e})]$
\item $\tau_1^S(f_{a,b})^p$
\end{itemize}

We now need to consider the relations among generators of the form $\tau_1^S(T_{\gamma}^p)\in \tau_1^S(D_p)$. We are interested in the case when $f \in D_p$ has trivial image in $\mathfrak{sp}_{2g}(\F_p)$. Recall $\Psi:Mod_{g,1}[p]\longrightarrow Sp_{2g}(\Z)[p]$ denotes the symplectic representation. Recall that $abel$ denotes the abelianization map of $Sp_{2g}(\Z)[p]$. Finally, recall $r\circ\tau_1^S(f) = abel\circ\Psi(f)$, where $r:Hom(\F_p^{2g},\el_2^S)\longrightarrow Hom(\F_p^{2g},\F_p^{2g})$. If $f\in ker(\Psi)$ then $f \in \I_{g,1}$. In this case, relations among $\tau_1^S(f)$ are already described above. It then suffices to consider the case when $f \in D_p$ and not an element of $\I_{g,1}$. If $\tau_1^S(f) = 0$ for such an $f$ then $\Psi(f) \in ker(abel)$. By Lemma \ref{ModPAbel} this implies that $\Psi(f) \in Sp_{2g}(\Z)[p^2]$. Hence, $\Psi(f)$ has a lift to $Mod_{g,1}[p]$ that is in $\left< D_{{p}^2}\right>$. Any two lifts of $\Psi(f)$ differ by multiplication with elements of $\I_{g,1} = ker(\Psi)$.

We may then take the following set of relations in the generating set $\tau_1^S(S)$ for our presentation of $image(\tau_1^S)$:
\begin{itemize}
\item $\tau_1^S(T_c)$, where $c$ is separating
\item $[\tau_1^S(f_{a,b}),\tau_1^S(f_{d,e})]$
\item $\tau_1^S(f_{a,b})^p$
\item $\tau_1^S(T_{\gamma}^p)^p$
\item $[\tau_1^S(f_{a,b}),\tau_1^S(T_{\gamma}^p)]$
\item Relations in $\tau_1^S(T_{\gamma}^p)$ sufficient to present the $\mathfrak{sp}_{2g}(\F_p)$ summand of the image of $\tau_1^S$
\end{itemize}
\noindent The first relators lift to separating twists. The second lift to elements of $\K_{g,1}$, which in turn is generated by separating twists. The third lift to $p^{th}$-powers of Dehn twists. The fourth lift to $p^{2}$-powers of Dehn twists. The fifth relations in the list lift to $p^{th}$-powers of bounding pair maps.

Now we consider the final relations. Recall $\mathfrak{sp}_{2g}(\F_p) \cong Sp_{2g}(\Z)[p]/Sp_{2g}(\Z)[p^2]$. Our generating set for $\mathfrak{sp}_{2g}(\F_p)$ consists of elements of the form $\tau_1^S(T_{\gamma}^p)$. There are two types of relations: the relations necessary to present $Sp_{2g}(\Z)[p]$ and the relations necessary to generate $Sp_{2g}(\Z)[p^2]$. Relations of the second kind lift by $\tau_1^S$ to words in $D_{p^2}$. For the sake of this proof we need not know the exact words since we have already taken all of $D_{p^2}$ to be in our generating set for $ker(\tau_1^S)$. Relations of the first kind lift by $\tau_1^S$ to a word $u$ in $D_p$ that is also an element of $\I_{g,1}$. Such an element $u$ satisfies $\tau_1^Z(u) = 0$ since $\tau_1^Z(D_p) = 0$, and $\tau_1^Z(u) = \tau_1(u)\text{ mod } p$ since $u \in \I_{g,1}$. Therefore, either $u\in ker(\tau_1)$, in which case it can be written as a product of separating twists, or $u$ is a product of $p^{th}$-powers of bounding pair maps.

Hence, $\I_{g,1}^S(2)$ is normally generated by separating twists, $p^{th}$-powers of bounding pair maps, and $p^2$-powers of Dehn twists. Again Lemma \ref{lemma:ConjTwists} shows this to actually be an honest generating set.\qedhere
\end{proof}


\section{Rational homology spheres}\label{Section:QHS}

Let $i:\Sigma_{g}\longrightarrow \mathbf{S}^3$ be a genus $g$ Heegaard surface in $\mathbf{S}^3$. We may obtain a 3-manifold by cutting $\mathbf{S}^3$ along $i(\Sigma_g)$ and regluing the resulting handlebodies $V$ and $W$ along their boundaries by an element $f\in Mod_{g,1}$. Here we think of the boundary component of $\Sigma_{g,1}$ as an embedded disk fixed by $f$. In this way we may think of $f$ as a mapping class of $\Sigma_g$ as well. This is done explicitly in the following way. We identify $V$ and $W$ with the standard handlebody $H_g$ via diffeomorphisms $f_V$ and $f_W$. Let $\phi$ be an orientation reversing diffeomorphism of $H_g$. Let $\overline{H}_g$ denote $\phi(H_g)$. We define our 3-manifold by $M_f = H_g\cup_{\phi\circ f} \overline{H}_g$. The gluing is done on the boundaries of our handlebodies by gluing $f_V(x) \in \partial V$ to $f_W\circ \phi\circ f(x) \in \partial \overline(W)$. We call $f \in Mod_{g,1}$ the Heegaard gluing map. We note that if $f$ is an element of $Mod_{g,1}[p]$ then $M_f$ is a rational homology 3-sphere (QHS). If $f$ is an element of $\I_{g,1}$ then $M_f$ is an integral homology 3-sphere (ZHS). 

Associated to a Heegaard splitting are two Lagrangian subspaces of $H_1(i(\Sigma_g))$, which
are defined as:
\begin{align*}
L_V &= ker(H_1(\partial V) \longrightarrow H_1(V)) \\
L_W &= ker(H_1(\partial W) \longrightarrow H_1(W))
\end{align*}
\noindent Here the maps are induced by the inclusions of the boundaries into the handlebodies. Note that $\partial H_g = \Sigma_g$ so that $H_1(\partial V) \cong H_1(\Sigma_g) \cong H_1(\partial W)$ by the diffeomorphisms $f_V$ and $f_W$.

We say that two gluing maps $f$ and $f'$ are equivalent if $f' = h_1\circ f \circ h_2$ for some 
$h_i \in Mod_{g,1}$ such that $h_1$ extends to a diffeomorphism $\overline{h}_1$ of the handlebody $V$ and $h_2$ extends to a 
diffeomorphism $\overline{h}_2$ of $W$. This means there is a diffeomorphism $F:M_f\longrightarrow M_{f'}$ that restricts to a diffeomorphism on each of the handlebodies. Namely, $F|_V = \overline{h}_1$ and $F|_W = \overline{h}_2$. By work of Birman we have a nice characterization of mapping classes $h$ that extend over 
handlebodies.

\begin{thm}[Birman, {\cite{Birman}}]\label{theorem:ExtendsOverBody}
Let $h \in Mod_{g}$, and suppose $h$ extends over the handlebody $V$. Let $\{a_1,...,a_g,b_1,...,b_g\}$ be a symplectic basis for $H_1(\Sigma_g)$ with representatives of $a_i$ bounding disks in $V$ and $b_i$ bounding disks in $W$. Then the image of $h$ in $Sp_{2g}(\Z)$ under the symplectic representation $\Psi$ is a matrix with $g \times g$ block form:

\[
\left( \begin{array}{cc}
* & * \\
0 & * \end{array} \right)
\]

\noindent If instead, $h$ extends over $W$ then $\Psi(h)$ is of the form:

\[
\left( \begin{array}{cc}
* & 0 \\
* & * \end{array} \right)
\]

\noindent That is, $\Psi(h)$ preserves one of the Lagrangian subspaces $L_V$ or $L_W$, respectively. 
\end{thm}

Consider the case when $f \in Mod_{g,1}[p]$. We have from the Mayer-Vietoris sequence afforded by this Heegaard splitting of $M$ that:

\[
H_1(M) \cong \frac{H_1(\Sigma_{g,1})}{\left\langle f(L_V), L_W \right\rangle}
\]

\noindent Since $M$ is a QHS we know that $|H_1(M)| = n < \infty $. We note that since $M$ is connected, the universal 
coefficient theorem gives us that $H_1(M;\F_p) = H_1(M) \otimes \F_p$. In particular, if $p$ is a prime such 
that $(p,n) = 1$ then $H_1(M;\F_p) = 0$. Hence, 
$H_1(\Sigma_{g,1};\F_p) \cong \left\langle L_W, f(L_V) \right\rangle \otimes \F_p$.

\begin{thm}
Let $M$ be a rational homology 3-sphere. Let $n = |H_1(M)|$. Then for any prime $p$ relatively prime to $n$ there is a mapping class $f \in Mod_{g,1}[p]$ for some $g$ such that $M = V \cup_f W$ is a Heegaard splitting for $M$.
\end{thm}

\begin{proof}
Pick a symplectic basis $\{a_1, b_1,...,a_g, b_g\}$ for $H_1(\Sigma_{g,1})$ such that $\{a_i\}$ is also a basis for $L_V$ and $\{b_i\}$ is a basis for $L_W$ as in Theorem \ref{theorem:ExtendsOverBody}. We know there is some Heegaard gluing map $f$ giving the QHS $M$. Let $p$ be as in the hypothesis, and
consider the matrix $\Psi(f)\ \text{mod}\ p = \Psi_p(f) \in Sp_{2g}(\F_p)$. Since $f$ is invertible $f^{-1}$ is 
defined and we may write $\Psi_p(f^{-1})$ as the $g\times g$ block matrix

\[ \left( \begin{array}{cc}
E & F \\
G & H \end{array} \right)\]

\noindent Since $H_1(M;\F_p) = 0$ we have that:
\[
span\{a_i, b_i\} = span\{f(a_i), b_i\} = span\{a_i, f^{-1}(b_i)\}.
\]
\noindent In particular, the map $f^{-1}_b$ on $H_1(\Sigma_{g,1};\F_p)/span\{a_i\}$ must have full rank. It follows that
the block $H$ is invertible. We consider a matrix $X$ of the form:

\[ 
X =  \left( \begin{array}{cc} Id_{g} & B' \\
0 & Id_{g}  \end{array} \right) \left( \begin{array}{cc} E & F \\
G & H  \end{array} \right) =  \left( \begin{array}{cc} E+B'G & F+B'H \\
G & H  \end{array} \right)
\]

We want to choose $B'$ such that $X$ lifts to a mapping class that extends over one of our handlebodies.
This is equivalent to either $F+B'H$ or $G$ being zero by Theorem \ref{theorem:ExtendsOverBody}.
$G$ is given so we must choose $B' = -FH^{-1}$ so that $F+B'H = 0$. This choice of $B'$ is defined since as
noted $H$ is invertible. Furthermore, since $F^TH=H^TF$ by virtue of $\Psi_p(f^{-1})$ being symplectic we have:
\begin{align*}
{B'}^T &= -{(H^{-1})}^TF^T \\
	&=  -{H^{-1}}^TH^TFH^{-1}\\
	&= -FH^{-1} \\
	&= B'
\end{align*}
\noindent Hence, $X$ is indeed a symplectic matrix, and we now have that $X\Psi_p(f)$ lifts to some $f'$ equivalent to $f$ acting trivially on $span\{a_i\}$ in $H_1(\Sigma_{g,1};\F_p)$.

To complete the proof let $Y = (X\Psi(f))^{-1}$. Since $X\Psi(f)$ has its lower left block zero, so does $Y$ 
(else $Id_g = 0_g$). In particular, $Y$ extends over the other of our handlebodies. Now $X\Psi(f)Y$ lifts to a
mapping class that is equivalent to $f$ and acts trivially on all of $H_1(\Sigma_{g,1};\F_p)$ as it is the identity 
matrix modulo $p$.\qedhere
\end{proof}

The case $p = 0$ gives a proof of the following corollary. The result was of course previously known, but to our knowledge a full proof has not yet been available in the literature.

\begin{cor}
If $M$ is an integral homology sphere then we may find an element $f \in \I_{g,1}$ for some $g$ such that $M = V \cup_f W$ is a Heegaard splitting.
\end{cor}

The above says that the Torelli groups $\I_{g,1}$ suffice to provide the necessary gluing maps to obtain any ZHS. 
Pitsch argues in \cite{Pitsch} that in fact the $\I_{g,1}(3)$ term of the Johnson filtration suffices to construct all ZHSs.
This allows us to offer a similar statement in the case of QHSs.

\begin{thm}
Let $M$ be a rational homology 3-sphere. Let $n = |H_1(M)|$. Then for any prime $p$ relatively prime to $n$ 
there is a mapping class $f \in \I_{g,1}(3)\cdot D_p$ for some $g$ such that $M = V \cup_f W$ is a 
Heegaard splitting for $M$. In particular, for prime $p \geq 5$ it follows that $f\in \I_{g,1}^Z(3)$.
\end{thm}

\begin{proof}
We are guaranteed that $M$ is as a Heegaard gluing my some $f \in Mod_{g,1}[p]$. We may write 
$f = f_1\circ f_2$ with $f_1\in \I_{g,1}$ and $f_2\in D_p$. By Pitsch's result $f_1$ is equivalent 
to some $f_1' \in \I_{g,1}(3)$. That is $f_1 = h_1f_1'h_2^{-1}$, and so $f = h_1f_1'h_2^{-1}f_2$. 
We rewrite $f$ as $h_1f_1'h_2^{-1}f_2h_2h_2^{-1}$. Since $f_2 \in D_p$ and $D_p$ is a normal subgroup 
of $Mod_{g,1}$ we have that $f$ is equivalent to $f_1'\circ f_2'$ with $f_1'\in\I_{g,1}(3)$ and $f_2'\in D_p$.
As both $\I_{g,1}(3)$ and $D_p$ are subgroups of $\I_{g,1}^Z(3)$ for prime $p \geq 5$, we obtain the final statement.\qedhere
\end{proof}


\end{document}